\documentclass{daj}
\usepackage{kbordermatrix}
\usepackage{amssymb}
\usepackage{amsmath}
\usepackage{amsthm}

\newtheorem{thm}{Theorem}[section]\theoremstyle{plain}
\newtheorem{theorem}[thm]{Theorem}\theoremstyle{plain}
\newtheorem{proposition}[thm]{Proposition}\theoremstyle{plain}
\newtheorem{lemma}[thm]{Lemma}\theoremstyle{plain}
\theoremstyle{plain}

\theoremstyle{plain}
\newtheorem{claim}[thm]{Claim}\theoremstyle{plain}
\newtheorem{corollary}[thm]{Corollary}\theoremstyle{plain}
\theoremstyle{plain}
\theoremstyle{plain}
\theoremstyle{plain}
\theoremstyle{plain}
\theoremstyle{plain}
\newtheorem{conjecture}[thm]{Conjecture}\theoremstyle{plain}
\theoremstyle{plain}

\DeclareMathOperator{\cl}{cl}
\DeclareMathOperator{\rank}{rank}

\newcommand{\sm}{{\setminus}}
\newcommand{\R}{\mathbb{R}}
\newcommand{\MC}{{\mathcal C}}
\newcommand{\MX}{{\mathcal X}}

\newcommand{\valD}{{\rm val}_{\rm D}}
\dajAUTHORdetails{%
  title = {Abstract 3-Rigidity and Bivariate $C_2^1$-Splines II: Combinatorial
Characterization}, %% please capitalize all significant words
  author = {Katie Clinch, Bill Jackson, and Shin-ichi Tanigawa},
  plaintextauthor = {Katie Clinch, Bill Jackson, and Shin-ichi Tanigawa},
  plaintexttitle = {Abstract 3-Rigidity and Bivariate C_2^1-Splines II: Combinatorial
Characterization}, %%  title without math or LaTeX
  runningtitle = {Abstract 3-Rigidity and Bivariate $C_2^1$-Splines II: Combinatorial
Characterization}, 
  runningauthor = {Katie Clinch, Bill Jackson, and Shin-ichi Tanigawa},
  copyrightauthor = {Katie Clinch, Bill Jackson, and Shin-ichi Tanigawa},
  keywords = {graph rigidity, rigidity matroid, bivariate spline, cofactor matroid,  matroid erection, free elevation},
}   %%% END \dajAUTHORdetails

\dajEDITORdetails{%
   year={2022},
   number={3},
   received={10 November 2019},   % received date: example: 7 January 2017
   published={11 May 2022},  % published date
   doi={10.19086/da.343692},       % XXX = number of paper, e.g. da006 for paper#6
}   %%% END \dajEDITORdetails

\begin{document}

\begin{frontmatter}[classification=text]
%% EDITOR: this will force the keywords to appear right after the Abstract.
%%   If the abstract is too long and would force the keywords off the
%%   front page, please comment out % [classification=text] above
%%   This way the keywords will be floated on the bottom of the first page
%%   even though the Abstract spills over to the next page.

%%% AUTHOR: Title goes here.  This line is optional.  You must use it
%%   if title has footnote attached or requires nontrivial typesetting,
%%   e.g., inclusion of linebreaks to force nice layout.
\title{Abstract 3-Rigidity and Bivariate $C_2^1$-Splines~II: Combinatorial
Characterization} %% please capitalize all significant words

%%% AUTHOR:
%%% List all authors. If you wish, place grant acknowledgements in \thanks.
%%% In brackets include a short tag for each author.
\author[katie]{Katie Clinch
%\thanks{Supported by...}
}
\author[bill]{Bill Jackson
%\thanks{Supported by...}
}
\author[shin]{Shin-ichi Tanigawa
%\thanks{Supported by...}
}

%%% AUTHOR: Abstract goes here
\begin{abstract}
We showed in the first paper of this series that the generic $C_2^1$-cofactor matroid is the unique maximal abstract $3$-rigidity matroid. In this paper we obtain a combinatorial characterization of independence in this matroid.
This solves the cofactor counterpart of the combinatorial characterization problem for the rigidity of generic 3-dimensional bar-joint frameworks.
We use our characterization to verify that the counterparts of conjectures of Dress
%, Tay and Whiteley 
(on the rank function) and  Lov{\'a}sz and Yemini (which suggested a  sufficient connectivity condition for rigidity) hold for the $C_2^1$-cofactor matroid. 

\end{abstract}
\end{frontmatter}

%%% AUTHOR: body of paper starts here
\section{Introduction}

We will consider a {\em $d$-dimensional (bar-joint) framework},  which is a pair consisting of a finite simple graph $G=(V,E)$ and a map $p:V\rightarrow \mathbb{R}^d$. It is {\em rigid} if every continuous motion of the vertices  of $G$ in $\mathbb{R}^d$ which preserves the lengths of the edges,  results in a framework which is congruent to $(G,p)$. The framework has the stronger property of being {\em infinitesimally rigid} if every length preserving infinitesimal motion of the vertices is induced by an infinitesimal isometry of $\mathbb{R}^d$. 
%(see Section~\ref{subsec:examples} for a formal definition).  
Maxwell \cite{Max} gave the following necessary condition for
% a $d$-dimensional framework to be  
{infinitesimal rigidity}: if $(G,p)$ is an
infinitesimally rigid $d$-dimensional framework with at least $d$ vertices, then $G$ has a spanning subgraph $H$ such that
\begin{itemize}
\item $|E(H)|=d|V(H)|-{d+1\choose 2}$, and
\item $|E(H')|\leq d|V(H')|-{d+1\choose 2}$ for any subgraph $H'$ of $H$ with $|V(H')|\geq d$.
\end{itemize}
%where $V(H)$ and $E(H)$ denote the sets of vertices and edges, respectively, in the subgraph $H$.
It is straightforward to  check that this condition is also sufficient when $d=1$.
A celebrated result of Pollaczek-Geiringer~\cite{P-G}, rediscovered by  Laman \cite{Lam}, shows that Maxwell's condition is also sufficient when $d=2$ {\em if $p$ is generic}, i.e., the set of coordinates in $p$ is algebraically independent over the rational field.
When $d\geq 3$, Maxwell's condition is no longer sufficient to imply infinitesimal rigidity, even for generic frameworks, and finding a combinatorial characterization for generic 3-dimensional rigidity  has been the central open problem in graph rigidity for many years, see for example~\cite{GSS93,Wsurvey}.
Maxwell's condition is known to be sufficient for generic 3-dimensional rigidity of some restricted classes such as triangulations of closed 2-surfaces (with/without holes)~\cite{CKT15,CKT18,F,FW}, graphs with no $K_5$-minor~\cite{N07}, and squares of  graphs~\cite{KT11}. We refer the reader to \cite{SW} for a recent survey on graph rigidity.
%, sparse or degree-bounded graphs~\cite{JJ05b}. 

Gluck~\cite{Glu} observed that the properties of rigidity and infinitesimal rigidity   coincide when $p$ is  generic,
and  are completely determined by the graph $G$ and the dimension $d$. This fact motivated Asimow and Roth~\cite{AR78} to define  a graph $G$ as being {\em rigid in $\mathbb{R}^d$} if   some, or equivalently every, generic $d$-dimensional framework $(G,p)$  is infinitesimally rigid.
The rigidity of a graph $G$ is determined by the rank of the {\em generic $d$-dimensional rigidity matroid ${\cal R}_d(G)$ of $G$} and
%defined on $E(G)$
%(see~Section~\ref{subsec:generic_rigidity} for the definition).
exploring the combinatorial structure of ${\cal R}_d(G)$ using  machinery from matroid theory is a common approach to attack problems in rigidity, see for example~\cite{TW84,GSS93,Wsurvey}. We  will build on this matroidal approach to obtain a combinatorial characterization of a matroid from the theory of bivariate splines
%, the {\em maximal abstract 3-rigidity matroid} and the {\em $C_2^1$-cofactor matroid}, 
which is conjectured to be equal to the generic 3-dimensional rigidity matroid.

%More specifically, we obtain a combinatorial characterization of two matroids, the {\em maximal abstract 3-rigidity matroid} and the {\em $C_2^1$-cofactor matroid}, which are conjectured to be equal to the generic 3-dimentional rigidity matroid.

%This paper built on this matroidal approach. More specifically, in this paper we study on the two matroids, the {\em maximal abstract 3-rigidity matroid} and the {\em $C_2^1$-cofactor matroid}, which are conjectured to be equal to the generic 3-dimentional rigidity matroid. We solves the combinatorial characterization problem for those two matroids. 

Graver~\cite{G91} defined  the class of {\em abstract $d$-rigidity matroids} on the edge set of the  complete graph $K_n$ using two fundamental properties of rigidity in $d$-space. The generic $d$-dimensional rigidity matroid ${\cal R}_d(K_n)$  is an example of an abstract $d$-rigidity matroid, and Graver~\cite{G91} conjectured that, for all $d\geq 1$, ${\cal R}_d(K_n)$ is the unique maximal matroid in the set of all abstract $d$-rigidity matroids on $E(K_n)$ with respect to the weak order of matroids. Graver showed that his conjecture is true for $d=1,2$ but it was subsequently shown to be false for $d\geq 4$. It remains open for $d=3$. (See Section~\ref{sec:rigidity} of this paper or \cite{CJT1} for a more detailed discussion.)

Whiteley~\cite{Wsurvey} found a new candidate for a unique maximal abstract $d$-rigidity matroid from approximation theory by taking the row matroid of the {\em $C^{s-1}_s$-cofactor matrix} of a(ny) generic 2-dimensional framework $(K_n,p)$. This is  the   $|E(K_n)|\times (s+1)n$ matrix in which sets of consecutive $(s+1)$ columns are associated with  vertices,  rows are associated with  edges,
and the row associated to the  edge $e=v_iv_j$ with $i<j$ is
\[
\kbordermatrix{
 & & v_i & & v_j & \\
 e=v_iv_j & 0\cdots 0 & D_{ij} & 0 \cdots 0 & -D_{ij} & 0\cdots 0
},
\]
where  $D_{i,j}=((x_i-x_j)^s, (x_i-x_j)^{s-1}(y_i-y_j),\dots, (x_i-x_j)(y_i-y_j)^{s-1}, (y_i-y_j)^s)\in \mathbb{R}^{s+1}$ when $p(v_i)=(x_i,y_i)\in \mathbb{R}^2$ for each $v_i\in V(K_n)$.
We refer to the row matroid  of this matrix as the {\em $C_s^{s-1}$-cofactor  matroid} and denote it by ${\cal C}_{s}^{s-1}(K_n)$. 
Whiteley~\cite{Wsurvey} showed that  ${\cal C}_{d-1}^{d-2}(K_n)$ is an %example of an 
abstract $d$-rigidity matroid, that   ${\cal C}_{d-1}^{d-2}(K_n)={\cal R}_{d}(K_n)$ 
%only 
when $d=1,2$ and that   ${\cal C}_{d-1}^{d-2}(K_n)\not\preceq {\cal R}_{d}(K_n)$ 
%only 
when $d\geq 4$. He conjectured further that  ${\cal C}_{d-1}^{d-2}(K_n)$  is the  unique  maximal abstract $d$-rigidity matroid on $E(K_n)$  for all $d\geq 1$. Note that Whiteley's conjecture holds when $d=1,2$ by the above mentioned result of Graver. In our first paper~\cite{CJT1} in this series, we verified Whiteley's conjecture for $d=3$ by showing that ${\cal C}_{2}^{1}(K_n)$  is the  unique maximal abstract $3$-rigidity matroid (see Theorem~\ref{thm:cofactor} below). 

Because of the strong similarity between rigidity matroids and cofactor matroids, Whiteley~\cite[page 55]{W90} also remarked that
finding a combinatorial characterization of 
independence in 
the generic $C_2^1$-cofactor  matroid may be
as challenging as the corresponding problem for the generic $3$-dimensional rigidity matroid, and went on to conjecture that these two matroids are equal in \cite[Conjecture 10.3.2]{Wsurvey}. 
This paper solves the characterization problem for the generic $C_2^1$-cofactor  matroid (and equivalently for the maximal abstract 3-rigidity matroid) by giving a co-NP type characterization for independence (we will see below that an NP type characterization follows immediately from the Schwartz-Zippel Lemma).
%as in the Pollaczek-Geiringer/Laman theorem for 2-dimensional rigidity). 
We then use our characterization to verify the cofactor counterpart of two long-standing conjectures on the generic 3-dimensional rigidity matroid. 

Dress gave two conjectures on the rank of the 3-dimensional generic rigidity matroid in the 1980's. The first conjecture, which appeared in \cite{DDH83}, suggested a co-NP type  characterization for independence, but was subsequently disproved by Jackson and Jord{\'a}n \cite{JJ05}. The second conjecture, which was given at a rigidity conference in Montreal in 1987, see \cite{CDT,GSS93,T99}, describes the rank function in terms of the `rigid clusters' of the underlying graph and still remains open. It would give a useful structural property of the rigidity matroid but would not obviously give rise to a co-NP type characterization.  We prove in Theorem~\ref{thm:dress} below that the corresponding conjecture holds for the generic $C_2^1$-cofactor matroid.
%(and equivalently for the maximal abstract 3-rigidity matroid). 
We also show that  the modified versions of Dress's first conjecture given in \cite{JJ06,JDIMACS} hold for this matroid.

%As for the sufficient connectivity, in \cite{LY82} 
Lov{\'a}sz and Yemini \cite{LY82} proved that 6-connectivity is sufficient to imply that  graphs are generically rigid in 2-dimensional space, and conjectured that 12-connectivity is sufficient for  rigidity in 3-space.  We prove in Theorem~\ref{thm:12conn} that the corresponding statement is true for  the generic $C_2^1$-cofactor matroid 
%(and equivalently for the  maximal abstract 3-rigidity matroid). 
%Our result hence implies that, if Graver's conjecture that the generic 3-dimensional rigidity matroid is the unique maximal abstarct 3-rigidity matroid holds, then Dress' second conjecture and Lov{\'a}sz-Yemini's sufficient connectivity conjecture also holds.

The main technical innovation in our work is to find a new kind  of  rank formula for a matroid.
One of the major difficulties in attacking the 3-dimensional rigidity problem is the lack of understanding of how to construct a matroid based on Maxwell's condition when $d\geq 3$.
It is well-known that, for $d=1,2$, the edge sets satisfying Maxwell's condition form an independent set family of a matroid.  Lov{\'a}sz and Yemini~\cite{LY82} showed that the structure of this matroid can be completely understood using the theory of intersecting submodular functions.
%Unfortunately this combinatorial structure
Currently, this theory does  not seem to apply when  $d\geq 3$. In particular, it is not clear how to use this theory to obtain a polynomial algorithm to evaluate our formula for the rank function of the $C_2^1$-cofactor matroid.
%Following the first conjecture of  Dress 
%%(see, \cite{DDH83,TW84}) 
%there have been several other attempts~\cite{T99,JJ06,JJ06,CS14} to modify Maxwell's condition when $d=3$, but none of these modifications are known to give the rank function of a matroid.

Our rank formula 
%for the generic $C_2^1$-cofactor matroid 
%is motivated by the more general context of matroid constructions.
was inspired by a matroid construction due to
Crapo~\cite{C70}. He defined an {\em erection} of a matroid as an inverse operation to truncation.
He showed that the set of all erections of a matroid $M$ forms a lattice under the weak order for matroids,  and defined the maximum matroid in this lattice to be the {\em free erection} of $M$.
A sequence of (free) erections starting from $M$ always terminates after a finite number of steps and
%if the ground set is finite,
we refer to a matroid obtained by 
such a sequence as a
%maximum sequence of free erections starting from $M_0$ as the  
{\em (free) elevation} of $M$.
Examples due to Brylawski~\cite[Figure 7.9]{B86} and the second two authors~\cite[Theorem 5.4]{JT} show that the free elevation may not be the {\em unique} maximal element in the poset of all elevations of $M$, but we show in Lemma \ref{lem:unique} that it is always {\em a} maximal element. %in this poset. 
%obtained by sequence of erections (which may not be free).
%Nevertheless in Theorem~\ref{thm:maximal} we prove that the free elevation of $M_0$ is the unique maximal matroid over all matroids that
%contains each non-spanning circuit of $M_0$ as a circuit under a certain mild assumption on $M_0$,
%Existing algorithms for constructing the independent (or cyclic) set family of the free erection of $M_0$ inspired us to 
We then obtain 
%a function which gives 
an upper bound  on the rank function of 
%the free 
all elevations of $M$ in terms of the non-spanning circuits of  $M$ %(Conjecture~\ref{conj:rank}). 
in Lemma \ref{lem:rank1}. Conjecture~\ref{conj:rank} states that this upper bound is tight if and only if the free elevation of $M$ is the unique maximal element in the poset of all elevations of $M$.
%, ordered under the weak order of matroids 
Our main result, Theorem \ref{thm:rank}, shows that the upper bound is tight for the generic $C_2^1$-cofactor matroid and verifies a special case of this conjecture.
% more generally that it is tight if and only if the free erection of $M$ is the unique maximal element in the poset of all elevations of $M$
%, ordered under the weak order of matroids 
%(Conjecture~\ref{conj:rank}). 

Our results are relevant to another long-standing open problem, the {\em polynomial identity testing problem for symbolic determinants} (or the {\em Edmonds problem}). In this problem, we are given a matrix $A$ with entries in $\mathbb{Q}[x_1,\dots, x_n]$ 
%for a sufficiently large (EXPLAIN?) field $\mathbb{F}$, 
and we are asked to decide whether the rank of $A$ over $\mathbb{Q}(x_1,\dots, x_n)$ is at least a given number.
The Schwartz-Zippel Lemma~\cite{S80,Z79} implies that this problem admits a randomized polynomial time algorithm,
but developing a deterministic polynomial time algorithm is a major open problem in theoretical computer science.
The same  lemma  also tells us that the problem is in the class {\rm NP},
%and a difficult problem is to show that
but it is not known whether it is in {\rm co-NP}.
One approach, pioneered by Tutte \cite{Tu}, Edmonds \cite{Ed} and Lov{\'a}sz \cite{L89}, is to give a good characterization for the rank of $A$
%, that is, characterizing the rank
by showing it is the minimum value for a certain combinatorial optimization problem.
Our result offers a new example of this approach.
Indeed  Lov{\'a}sz~\cite[Section 5]{L89}   states that the generic 3-dimensional rigidity problem %given at the beginning
is an important special
%unsolved
case of  the polynomial identity testing problem. Our technique solves the closely related problem for generic $C_2^1$-cofactor matrices.
%maximal abstract 3-rigidity matroid
%and the 
%${\cal C}_2^1(K_n)$.

The research direction of this paper was motivated by the fundamental papers of Graver \cite{G91} and Whiteley \cite{Wsurvey}, and more recently by a talk given by Meera Sitharam at  a BIRS workshop
%: Advances in Combinatorial and Geometric Rigidity Theory
 in 2015 on her joint work with Jialong Cheng and Andrew Vince, see \cite{S15}, in which
%The most relevant work to our matroid construction is a recent preprint on the `maximum matroid for a graph' due to Meera Sitharam and Andrew Vince~\cite{V}.
she described a recursive procedure for constructing the  closure operator in a maximal matroid on the edge set of a 
 complete graph  in which every $K_5$-subgraph is a  circuit. 
% \footnote{They have recently released a preprint on arXiv \cite{SV} in which they claim that  their construction can be extended to show that  there exists a unique maximal matroid for any subgraph of a complete graph but this is false is false. The status of this assertion when $G$ is complete is unclear. Our paper \cite{CJT1} gives an independent proof that the assertion holds when $G=K_5$.
%%}
% Their approach is different to ours. In particular their construction is based  on the matroid closure axioms rather than the  theory of matroid erections.
% Their proof is based on a recursive algorithm for constructing the closure of a set in this matroid,
%and, as they explicitly state in their paper, their analysis is substantially different from the matroid erection approach.
%Nevertheless,
%%we should mention that
%our research direction was largely inspired by a talk on the result given by Meera Sitharam at  a BIRS workshop
%%: Advances in Combinatorial and Geometric Rigidity Theory
% in 2015, see \cite{S15}.

Our paper is structured as follows. In Section \ref{sec:erections} we review the theory of matroid erections, using a primal approach instead of the traditional dual approach. We introduce the free elevation 
%$M^\uparrow$ 
of a matroid $M$ in Section \ref{sec:c-matroid} and prove two key results:  Lemma \ref{lem:rank1} gives an upper bound on the rank of 
%the free 
any elevation of $M$;
%$M^\uparrow$;  
Lemma \ref{lem:covering} shows that if every element of 
$M$ is contained in a non-spanning circuit of $M$ then every cyclic flat
of 
%$M^\uparrow$ 
the free elevation of $M$ is the union of non-spanning circuits of $M$. We describe the family of abstract $d$-rigidity matroids and $C_{d-1}^{d-2}$-cofactor matroids in Section \ref{sec:rigidity}. We obtain our main result, Theorem \ref{thm:rank}, which gives a polynomially verifiable characterization of the rank function of the maximal abstract 3-rigidity matroid ${\cal C}_{2}^{1}$ in Section \ref{sec:char}.
We give two alternative expressions for the rank function of ${\cal C}_{2}^{1}$ in Section \ref{sec:shell} and use these to obtain sufficient connectivity conditions for the ${\cal C}_{2}^{1}$-matroid of a graph to have maximum possible rank in Section \ref{sec:con}. We close with some  open problems and remarks in Section \ref{sec:close}.

\section{Matroid erections}\label{sec:erections}
Matroid erection, introduced by Crapo~\cite{C70}, is a key tool in this paper.
We give a detailed exposition of matroid erection in this preliminary section for the benefit of readers who are unfamiliar with the topic. For an introduction to the concepts below, see \cite{O11}.

Given a matroid $M$, we use $E_M$ to denote its ground set, ${\rm cl}_M$ to denote its closure operator, and $r_M$ to denote its rank function. We will often suppress the subscript $M$ when it is obvious which matroid we are referring to.
For $X\subseteq E$, $M|_X$ denotes the restriction of $M$ to $X$ and $M- X$ denotes $M|_{E\sm X}$. The contracted matroid $M/X$ is the matroid on $E\sm X$ in which a set $F\subseteq E\sm X$ is independent if and only if $F\cup X$ is independent in $M$. When $X=\{e\}$, we simplify the notation for matroid deletion and contraction to $M-e$ and $M/e$, respectively.

%The {\em rank of $X$} is the size of a maximum independent subset of $X$,
%and the rank of $E(M)$ is also referred to as  the {\em rank of $M$}.
A set $X\subseteq E$ is said to be  {\em spanning} if ${\rm cl}_M(X)=E$ and to be a {\em flat} if ${\rm cl}_M(X)=X$. A {\em hyperplane} of $M$ is a flat $F$ with $r(F)=r(E)-1$. 
The poset of all flats ordered by set inclusion forms a geometric lattice by setting the meet and join of two flats $F_1$ and $F_2$ to be
 $F_1\cap F_2$, and 
 %the smallest flat containing 
 $\cl_M(F_1\cup F_2)$, respectively.
A pair $X,Y$ of  subsets of $E$ is said to be {\em modular} if
$r_M(X)+r_M(Y)=r_M(X\cap Y)+r_M(X\cup Y)$.
The {\em dual} of
%a matroid
$M$  is denoted by $M^*$.
The {\em weak order} on a set of all matroids with the same ground set $E$ is the partial order in which $M_1\preceq M_2$ if every independent set of $M_1$ is independent in $M_2$.

\paragraph{One-point extensions and elementary quotients.}
Given two matroids $M$ and $P$, we say that $P$ is a {\em one-point extension} of $M$ if $M=P- p$ for some $p\in E_P$.
The structure of one-point extensions can be understood by introducing the concept of modular cuts.
A family ${\cal F}$ of flats of $M$ is said to be a {\em modular cut} of $M$ if
it is up-closed in the lattice of flats and $X\cap Y\in {\cal F}$ for all modular pairs $X, Y$ in ${\cal F}$.
Given a modular cut ${\cal F}$ of $M$, we can define a matroid $P$ on $E_M+p$ as follows. For all $X\subseteq E_M$ we put
$r_P(X)=r_M(X)$,  and
\[
r_P(X+p)=\begin{cases}
r_M(X) & \mbox{if ${\rm cl}_M(X)\in {\cal F}$}, \\
r_M(X)+1 & \mbox{if ${\rm cl}_M(X)\not\in {\cal F}$.}
\end{cases}
%\qquad (X\subseteq E(N)).
\]
One can easily check that $r_P$ is indeed a matroid rank function.
We will denote the one-point extension of $M$ with respect to the modular cut $\cal F$ by $M+_{\cal F} p$.
Every one-point extension of $M$ can be uniquely constructed in this manner.
More precisely, the map ${\cal F}\mapsto M+_{\cal F} p$ is a bijection between the set of modular cuts of $M$ and the set of one-point extensions of $M$.
(See~\cite{O11} for more details.)

Given a modular cut $\cal F$ of $M$, the matroid $N=(M+_{\cal F} p)/p$,
is called the {\em elementary quotient} of $M$ with respect to ${\cal F}$.
Observe that $E_M=E_N=:E$ and, for all $X\subseteq E$,
\begin{equation}\label{eq:quotient}
r_N(X)=\begin{cases}
r_M(X)-1 & \mbox{if ${\rm cl}_M(X)\in {\cal F}$}, \\
r_M(X) & \mbox{if ${\rm cl}_M(X)\not\in {\cal F}$.}
\end{cases}
%\qquad (X\subseteq E(N)).
\end{equation}

\paragraph{Elementary lifts.}
 Given two matroids $M$ and $N$, we say that $N$ is  an {\em elementary lift} of $M$ if $M$ is an elementary quotient of $N$ i.e.~$N=P-p$ and $M=P/p$ for some $1$-point extension $P$ of $N$. In this case we can use matroid duality to deduce that $M^*=P^*-p$ and $N^*=P^*/p$. Hence $N$ is an elementary lift of $M$ if and only if $N^*$ is an elementary quotient of $M^*$.
The correspondence between elementary quotients and modular cuts now gives us a bijection between the set of elementary lifts of $M$ and the set of modular cuts of $M^*$. It will be helpful to describe this bijection in terms of the primal matroid $M$ rather than its dual $M^*$.

A set $X\subset E$ is said to be {\em cyclic} in $M$ if it is the union of circuits of $M$.
For our purposes, it will make sense to consider the empty set as a cyclic set.
Let ${\rm cyc}_N(X)$ be the largest cyclic subset of $X$, i.e., the set obtained from $X$ by removing the coloops in $M|_X$. The poset of all cyclic sets in $M$ ordered by set inclusion forms a lattice by setting the join and the meet of $C_1$ and $C_2$ to be  $C_1\cup C_2$ and ${\rm cyc}_M(C_1\cap C_2)$, respectively.

We say that a family ${\cal C}$ of cyclic subsets of $E$ is  a {\em modular cyclic family} if
%\begin{itemize}
%\item every $C\in {\cal C}$ is cyclic;
%\item
${\cal C}$ is down-closed (in the lattice of cyclic sets) and,
%\item
for every modular pair $X, Y$ in $\cal C$, $X\cup Y\in {\cal C}$. Note that since modular cyclic families are down-closed and we consider $\emptyset$ to be cyclic, we have $\emptyset\in {\cal C}$ for every modular cyclic family $\cal C$. 
%\end{itemize}

The following result gives a bijection between the modular cuts
of $M^*$
% in the dual matroid
and the modular cyclic families in $M$.
%the primal.

\begin{proposition}\label{prop:modular_cyclic}
Let  ${\cal C}$ be a family of subsets of $E$.
Then ${\cal C}$ is a modular cyclic family in $M$ if and only if
 $\{E\setminus C: C\in {\cal C}\}$ is a modular cut in $M^*$.
\end{proposition}
\begin{proof}
Recall that $X$ is a circuit in $M$ if and only if $E\setminus X$ is a hyperplane in $M^*$.
Since every flat is the intersection of hyperplanes,  $X$ is cyclic in $M$ if and only if $E\setminus X$ is a flat in $M^*$.
In addition, a direct computation shows that  $X, Y$ is a modular pair in $M$ if and only if  $E\setminus X, E\setminus Y$ is a modular pair in $M^*$.
 \end{proof}
 Proposition~\ref{prop:modular_cyclic} and the preceding discussion give  a bijection between   the set of elementary lifts of $M$ and the set of modular cyclic families in $M$, and allow us to define $N$ to be the {\em elementary lift of $M$ with respect to the modular cyclic family $\cal C$ of $M$} if and only if $N^*$ is the {elementary quotient of $M^*$ with respect to the modular cut
 $\{E\setminus C: C\in {\cal C}\}$ of $M^*$.}
 Moreover, we can use (\ref{eq:quotient}), Proposition~\ref{prop:modular_cyclic}, and the fact that ${\rm cyc}_M(X)=E\setminus {\rm cl}_{M^*}(E\setminus X)$ for $X\subseteq E$ to deduce that
 the elementary lift $N$ of $M$ with respect to
the modular cyclic family
 ${\cal C}$ in $M$ has the following rank function:
 \begin{equation}
 \label{eq:lift}
r_N(X)=\begin{cases}
r_M(X) & \mbox{if ${\rm cyc}_M(X)\in {\cal C}$} \\
r_M(X)+1 & \mbox{if ${\rm cyc}_M(X)\not \in {\cal C}$}
\end{cases}
\qquad (X\subseteq E(N)).
\end{equation}

This formula implies that the above mentioned bijection is an (order reversing) isomorphism between the lattice of elementary lifts of $M$, ordered by the weak order of matroids, and the lattice of modular cyclic families in $M$, ordered by inclusion.  (The meet and join of two cyclic families in this lattice are given by  the intersection of the two families, and the smallest modular cyclic family containing their union, respectively.)

\paragraph{Matroid erections.}
Let $M$ and $N$ be matroids on the same ground set $E$ with $r_M(E)=k\leq r_N(E)$.
We say that $M$ is the
%{ \em $i$-th truncation}
{\em truncation of $N$ to rank $k$} if $r_M(X)=\min\{r_N(X), k\}$ for all $X\subseteq E(M)$.
If $k=r_N(E)-1$, $M$ is simply called the {\em truncation} of $N$.
The inverse operation to truncation was used by Crapo~\cite{C70} and Knuth~\cite{K75} to recursively generate all the matroids on a given ground set  from the rank zero matroid on this set. Following Crapo, we say that $N$ is an {\em erection} of $M$ if $M$ is the truncation of $N$.
For a technical reason, $M$ is also considered to be  an erection of itself,  and is referred to as   the {\em trivial} erection.

Crapo~\cite{C70} showed that the set of all erections of a matroid $M$ forms a lattice under the weak order,
where the bottom element corresponds to the trivial erection.
The top element in this lattice is called the {\em free erection} of $M$.
Las Vergnas~\cite{L76} and Nguyen~\cite{N79} independently gave a characterization of the free erection of $M$.  Duke~\cite{D87} subsequently gave a clearer exposition in terms of one-point extensions of the dual matroid $M^*$.
We shall describe Duke's approach in terms of the primal matroid $M$.

Observe first that, if $M$ is a truncation of $N$, then (\ref{eq:quotient}) implies that $M$ is the elementary quotient of $N$ with respect to the modular cut ${\cal  F}=\{E\}$. This in turn implies that $N$ is  a special kind of elementary lift of $M$. Our next result characterizes which elementary lifts are erections.

%A question is which elementary lift is an erection.
\begin{theorem} {\rm (A primal version of \cite[Lemma 3.1]{D87}.)}\label{thm:erection}
Let ${\cal C}$ be a modular cyclic family in a matroid $M$ and
$N$ be the elementary lift of $M$ with respect to ${\cal C}$.
Then $N$ is an erection of $M$ if and only if ${\cal C}$ contains all cyclic flats of $M$, with the possible exception of $E$.
\end{theorem}
\begin{proof}
We first assume that $N$ is an erection of $M$. Then $M$ is the truncation of $N$. Let $C$ be a  cyclic flat of $M$ with $C\not\in \cal C$.  By (\ref{eq:lift}) we have $r_N(C)=r_M(C)+1$. Since $M$ is the truncation of $N$ this gives $r_M(C)=r_M(E)$. Since $C$ is a flat of $M$, this implies that $C=E$.

We next assume that ${\cal C}$ contains all cyclic flats of $M$, with the possible exception of $E$. Suppose
$F$ is a flat of $M$  with $r_M(F)<r_M(E)$.  Then ${\rm cyc}_M(F)$ is a cyclic flat so ${\rm cyc}_M(F)\in {\cal C}$.
% since ${\cal C}$ is down-closed.
By (\ref{eq:lift}), $r_N(F)=r_M(F)$. It follows that all sets $X$ with $r_N(X)=r_M(X)+1$  have $r_M(X)=r_M(E)$. Hence $M$ is the truncation of $N$.
\end{proof}
%\item The set of elementary lifts is in a one-to-one correspondence with the set of modular cyclic families;
We saw above that the lattice of all elementary lifts of $M$ is isomorphic to the lattice of modular cyclic families of $M$. Theorem~\ref{thm:erection} enables us to determine the restriction of this isomorphism to the lattice of erections of $M$. Given a family ${\cal C}_0$ of cyclic sets in $M$, we  define its {\em modular cyclic closure} as:
\[
\overline{{\cal C}}_0=
%\bigwedge 
\bigcap\{ {\cal C}: {\cal C} \text{ is a modular cyclic family of $M$ with } {\cal C}_0\subseteq {\cal C}\}.
\]
Let ${\cal CF}_M$ be the family of all non-spanning cyclic flats of $M$,
% except for $E(N)$,
and ${\cal C}_M$ be the family of all cyclic sets in $M$.
Then Theorem~\ref{thm:erection} implies that the sublattice $[\overline{{\cal CF}}_M, {\cal C}_M]$ in the lattice of modular cyclic families corresponds to the lattice of erections of $M$, where $\overline{{\cal CF}}_M$ corresponds to the free erection and ${\cal C}_M$ corresponds to the trivial erection. In particular we have:
\begin{corollary}
A matroid $M$ has  no non-trivial erection if and only if $\overline{{\cal CF}}_M={\cal C}_M$, i.e.,
the only modular cyclic family of $M$ containing ${\cal CF}_M$ is the family of all cyclic sets in $M$.
\end{corollary}
By (\ref{eq:lift}) we also have the following explicit rank formula for the free erection of $M$.
\begin{corollary} \label{cor:free_erection}
Let $N$ be the free erection of $M$.
Then
\begin{equation}
 \label{eq:erection}
r_N(X)=\begin{cases}
r_M(X) & \mbox{ if ${\rm cyc}_M(X)\in \overline{{\cal CF}}_M$} \\
r_M(X)+1 & \mbox{ if ${\rm cyc}_M(X)\not \in \overline{{\cal CF}}_M$}
\end{cases}
\qquad (X\subseteq E).
\end{equation}
\end{corollary}

For a set $S$ of elements in a lattice, let  $S^{\downarrow}$ be the lower closure of $S$.
In order to use Corollary~\ref{cor:free_erection} to determine the rank function for the free erection of $M$, we need to be able to compute $\overline{{\cal CF}}_M$ from ${\cal CF}_M$.
%\item For a family ${\cal F}$ of cyclic sets, let  ${\cal F}^{\uparrow}$ be the upper closure, i.e., ${\cal F}^{\uparrow}=\{F: \exists F'\in {\cal F}, F'\subseteq F\}$.
%\item It is known that the following procedure gives the smallest modular cut that contains a given set ${\cal F}$ of flats:
%\begin{itemize}
%\item ${\cal F}_0:={\cal F}$
%\item ${\cal F}_i:={\cal F}_{i-1}\cup \{X\cap Y \mid  X, Y\in {\cal F}_{i-1}^{\uparrow}:  \text{ $X$ and $Y$ form a modular pair}\}$
%\item If ${\cal F}_i={\cal F}_{i-1}$, then ${\cal F}_i^{\uparrow}$ is the smallest modular cut that contains ${\cal F}$.
%\end{itemize}
An  algorithm for constructing  the smallest modular cut containing a given set of flats is known (see, e.g.,~\cite[page 367]{D87}).
Dualizing this algorithm, we may use the following procedure  to compute $\overline{{\cal CF}}_M$ from ${\cal CF}_M$.

\medskip

\noindent{\bf Algorithm 1}

$\bullet$ Initialize ${\cal S}_0:={\cal CF}_M$.

$\bullet$ Repeatedly construct ${\cal S}_i$ from ${\cal S}_{i-1}$ by
\\[-3mm]
\[
{\cal S}_i:={\cal S}_{i-1}\cup \{X\cup Y \mid X, Y\in {\cal S}_{i-1}^{\downarrow}:  \text{ $X$ and $Y$ form a modular pair in $M$} \}.
\]
\\[-9mm]

$\bullet$ If ${\cal S}_i={\cal S}_{i-1}$, then ${\cal S}_i^{\downarrow}$ is $\overline{{\cal CF}}_M$.

\medskip

%\begin{itemize}
%\item Initialize ${\cal S}_0:={\cal CF}_M$.
%\item Repeatedly construct ${\cal S}_i$ from ${\cal S}_{i-1}$ by
%\[
%{\cal S}_i:={\cal S}_{i-1}\cup \{X\cup Y \mid X, Y\in {\cal S}_{i-1}^{\downarrow}:  \text{ $X$ and $Y$ form a modular pair in $M$} \}.
%\]
%\item If ${\cal S}_i={\cal S}_{i-1}$, then ${\cal S}_i^{\downarrow}$ is $\overline{{\cal CF}}_M$.
%\end{itemize}

Let $M$ be a matroid on a finite ground set,
and let $M=M_0, M_1, M_2, \dots, M_k$ be a
%maximum
sequence of matroids starting from $M$ such that
$M_i$ is a non-trivial %free
erection of $M_{i-1}$ for all $1\leq i\leq k$ and $M_k$ has no non-trivial erection.
Since each non-trivial  erection increases the rank by one, and the rank is bounded above by $|E|$,
the length of any such sequence is bounded.
We will refer to the last matroid $M_k$ in such a sequence as an {\em elevation of} $M$.
%It follows immediately from this definintion, that every elevation of %$N$ will have no non-trivial erection.
The {\em free elevation of $M$} is the elevation we obtain from $M$ by taking a maximal sequence of non-trivial free erections.

\section{Matroids generated by a set of circuits}\label{sec:c-matroid}
%\section{${\cal C}$-matroids}\label{sec:c-matroid}
Let ${\cal C}$ be a family of subsets of a finite set $E$.
A matroid $M$ on $E$ is said to be a  {\em ${\cal C}$-matroid} if every member of ${\cal C}$ is a circuit in $M$.
We will be primarily concerned with matroids constructed by a sequence of free erections from a given matroid.
 More specifically, let $M_0$ be a matroid on a finite set $E$ and ${\cal C}_0$ be the family of non-spanning circuits in $M_0$. Then every  matroid obtained from $M_0$ by a sequence of erections is a ${\cal C}_0$-matroid and, in particular, the free elevation of $M_0$ is a ${\cal C}_0$-matroid.
We will derive some properties of the free elevation of $M_0$, which will be crucial tools in the proof of our main theorem.

\subsection{Maximality in the weak order}
Let ${\cal C}$ be a family of subsets of a finite set $E$.
%We consider the {\em weak order} on the family of ${\cal C}$-matroids on $E$,
%that is,  in the weak order $N\preceq M$ holds for two matroids $N, M$ if any independent set in $M$ is independet in $N$.
We say that a matroid $M$ is a {\em maximal} ${\cal C}$-matroid if it is maximal in the weak order on the family of all ${\cal C}$-matroids on $E$.

%Let $M_0$ be a matroid on a finite set $E$ and ${\cal C}_0$ be the family of all non-spanning circuits in $M_0$. We will show that the  free elevation of $M_0$ is a  maximal  ${\cal C}_0$-matroid on $E$.
%%As the free elevation $M_0^*$ of $M_0$ is obtained from $M_0$ by applying free erections as much as possible,
%It is tempting to guess that the free elevation of $M_0$ will be a maximal  ${\cal C}_0$-matroid but
%this is not true in general -- Brylawski~\cite[Figure 7.9]{B86} gives a counterexample.
%We first show that, if there is a unique maximal  ${\cal C}_0$-matroid,  then this unique maximal  ${\cal C}_0$-matroid will be the  free elevation of $M_0$.

\begin{lemma}\label{lem:unique}
Let $M_0$ be a matroid on $E$ and ${\cal C}_0$ be the family of all non-spanning circuits in $M_0$.
Then the free elevation of $M_0$ is a  maximal ${\cal C}_0$-matroid on $E$.
\end{lemma}
\begin{proof}
Let $k$ be the rank of $M_0$ and $M$ be the free elevation of $M_0$.
Suppose for a contradiction that $M\prec N$ for some
${\cal C}_0$-matroid $N$ on $E$.
% with $M\neq N$.  
We first prove:
\begin{claim}\label{claim:unique}
The truncation of $N$ to rank $k$ is equal to $M_0$.
\end{claim}
\begin{proof}
Let $N_0$ be the truncation of $N$ to rank $k$. The facts that $N$ is a ${\cal C}_0$-matroid and $M\prec N$  imply that ${\cal C}_0$ is the set of non-spanning circuits of $N_0$. This in turn implies that a set $X\subseteq E$ with $|X|=k$ is dependent in $M_0$ if and only if it is dependent in $N_0$, and hence that $M_0$ and $N_0$ have the same set of bases.
%
%Consider any independent set $I$ in $N'$.
%Then its size is at most $k$.
%Hence if $I$ is dependent in $M_0$ then it contains a circuit of rank at most $k-1$,
%which belongs to ${\cal C}_0$. Hence $I$ would be dependent as every set in ${\cal C}_0$ is also a circuit in $N'$, a contradiction.
%
%The same argument also shows that any independent set in $M_0$ is independent in $N'$, implying the claim.
\end{proof}

%We first show that, if there is a unique maximal  ${\cal C}_0$-matroid,  then this unique maximal  ${\cal C}_0$-matroid will be the  free elevation of $M_0$.

%\begin{lemma}\label{lem:unique}
%Let $M_0$ be a matroid on a finite set $E$ and let ${\cal C}_0$ be the family of all non-spanning circuits in $M_0$.
%Suppose that there is a unique  maximal ${\cal C}_0$-matroid $N$ on $E$.
%Then $N$ is  the free elevation  of $M_0$.
%\end{lemma}
%\begin{proof}
%Suppose that
%%the maximal ${\cal C}_0$-matroid $N$ on $E$ is distinct from
%$N$ is not the free elevation of $M_0$. Let $k$ be the rank of $M_0$. We first prove:
%\begin{claim}\label{claim:unique}
%The truncation of $N$ to rank $k$ is equal to $M_0$.
%\end{claim}
%\begin{proof}
%Let $N_0$ be the truncation of $N$ to rank $k$. The hypothesis that $N$ is a maximum ${\cal C}_0$-matroid implies that ${\cal C}_0$ is the set of non-spanning circuits of $N_0$. This in turn implies that a set $X\subseteq E$ with $|X|=k$ is dependent in $M_0$ if and only if it is dependenent in $N_0$, and hence that $M_0$ and $N_0$ habe the same set of bases.
%%
%%Consider any independent set $I$ in $N'$.
%%Then its size is at most $k$.
%%Hence if $I$ is dependent in $M_0$ then it contains a circuit of rank at most $k-1$,
%%which belongs to ${\cal C}_0$. Hence $I$ would be dependent as every set in ${\cal C}_0$ is also a circuit in $N'$, a contradiction.
%%
%%The same argument also shows that any independent set in $M_0$ is independent in $N'$, implying the claim.
%\end{proof}

Claim~\ref{claim:unique} implies that $M_0$ can be obtained from $N$ by a sequence of truncations, and hence that %(decreasing the rank one by one).
%Reversing the sequence,
$N$ can be obtained from $M_0$ by a sequence of erections, say
$M_0=N_0, N_1, \dots, N_{\ell}=N$. Let $M_0, M_1, \dots, M_{m}=M$ be the sequence of free erections which construct $M$ from $M_0$.
% be this sequence of erections from $M_0$ to $N$.
Since $N\neq M$,  we can choose a smallest possible $i$ such that $N_{i}\neq M_i$. Then $i\geq 1$, $N_j=M_j$ for all $0\leq j\leq i-1$ and $N_i$ is not the free erection of $M_{i-1}$.
Since  $M_i$ is the maximum element in the  lattice of all erections of $M_{i-1}$,   there is a set $X\subseteq E$ that is dependent in $N_{i}$ and independent in  $M_i$.
The set $X$ remains independent in  $M$ but will be dependent in $N$ as $\rank_{N_i}(X)<\rank_{M_i}(X)\leq \rank(M_i) =\rank(N_i)$ and        
$N_{i}$ is obtained from $N$ by truncations.
This contradicts the hypothesis that $M\prec N$.
\end{proof}

%Since each free erection is the unique maximal element in the lattice of all erections, it is tempting to guess that the free elevation of $M_0$ will be the unique maximal  ${\cal C}_0$-matroid.  
Since the free erection of a matroid $M_0$ is the unique maximal element in the lattice of all erections of $M_0$, it is tempting to guess that the free elevation of $M_0$ will be the unique maximal  element in the poset of all matroids we can obtain from $M_0$ by taking sequences of erections.
Sadly
this is not true in general -- counterexamples are given in \cite[Figure 7.9]{B86} and \cite[Theorem 5.4]{JT}.

\subsection{Rank upper bound}
We next obtain an upper bound of the rank of any ${\cal C}$-matroid.
Given a sequence of circuits  $(C_1,\dots, C_k)$ in a matroid $M$, we put
$C_{\leq i}=\bigcup_{j=1}^{i} C_j$ for all $1\leq i\leq k$ and  $C_{\leq 0}=\emptyset$.
The sequence $(C_1,\dots, C_k)$ is said to be {\em proper} if $C_i\not\subseteq C_{\leq i-1}$ for all $2\leq i\leq k$, and it is said to be a {\em ${\cal C}$-sequence} (for a family  of circuits ${\cal C}$) if each $C_i$ belong to ${\cal C}$.
The following lemma is fundamental to our characterization of the rank function of the maximal abstract 3-rigidity matroid.
\begin{lemma}
\label{lem:rank1}
Let $M$ be a matroid
on a finite set $E$ with rank function $r$
and ${\cal C}$ be a family of circuits in $M$.
Then for any $X\subseteq E$ and any proper ${\cal C}$-sequence $(C_1,\dots, C_t)$,
\begin{equation}\label{eq:upperbound}
r(X)\leq |X\cup C_{\leq t}|-t.
\end{equation}
Furthermore, if equality holds in {\rm (\ref{eq:upperbound})}, then $C_{\leq t}\subseteq \cl(X)$ and each $e\in X\sm C_{\leq t}$ is a coloop of $M|_X$.
\end{lemma}
\begin{proof}

We  first use  induction on $j$ to show that, for all $1\leq j\leq t$,
\begin{equation}
\label{eq:rank11}
r(C_{\leq j})\leq \sum_{i=1}^j (|C_i\setminus C_{\leq i-1}|-1).
\end{equation}

The base case when $j=1$ holds since $C_1$ is a circuit.

Suppose $j>1$. As the sequence is proper,
$C_j\cap C_{\leq j-1}$ is a proper subset of $C_j$, which is independent.
Hence its rank is equal to its cardinality.
Thus
\begin{align*}
r(C_{\leq j})&\leq r(C_{\leq j-1})+r(C_j)-r(C_{\leq j-1}\cap C_j)  \qquad (\text{by submodularity}) \\
&= r(C_{\leq j-1})+|C_j|-1-|C_{\leq j-1}\cap C_j|  \\
&\leq \sum_{i=1}^j (|C_i\setminus C_{\leq i-1}|-1) \qquad (\text{by induction})
\end{align*}
and (\ref{eq:rank11}) holds.

Putting $j=t$ in (\ref{eq:rank11}) gives $r(C_{\leq t})\leq |C_{\leq t}|-t$.
We can now use submodularity and the monotonicity of $r$ to deduce that
\[
r(X)\leq r(X\setminus C_{\leq t})+r(C_{\leq t})
\leq |X\setminus C_{\leq t}|+|C_{\leq t}|-t=|X\cup C_{\leq t}|-t.
\]
This completes the proof of the first part of the lemma.

To prove the second part, we assume that $r(X)=|X\cup C_{\leq t}|-t$ for some $X\subseteq E$ and some proper $\cal C$-sequence $(C_1,\ldots,C_t)$.
If $e\in C_{\leq t}$
then, by the first part of the lemma,
%$r(X+e)\leq {\rm val}(X+e, {\cal C})={\rm val}(X, {\cal C})=r(X)$,
$r(X+e)\leq |(X+e)\cup C_{\leq t}|-t=|X\cup C_{\leq t}|-t=r(X)$,
and hence  $e\in \cl(X)$. Similarly,
if $e\in X\setminus C_{\leq t}$
then, by the first part of the lemma,
%$r(X-e)\leq {\rm val}(X-e, {\cal C})={\rm val}(X, {\cal C})-1=r(X)-1$,
$r(X-e)\leq  |(X-e)\cup C_{\leq t}|-t= |X\cup C_{\leq t}|-t-1=r(X)-1$,
and hence  $e$ is a coloop of
$M|_X$.
\end{proof}

Let $M_0$ be a matroid on $E$ and
%on a finite set $E$ and let
${\cal C}_0$ be the family of non-spanning circuits in $M_0$.
%, and $M$ be a matroid obtained from $M_0$ by a sequence of erections.
Then, for any ${\cal C}_0$-matroid $M$ on $E$,  Lemma~\ref{lem:rank1} implies that the function $f_{{\cal C}_0}: 2^E\rightarrow \mathbb{Z}$ defined by
\[
f_{{\cal C}_0}(X)=\min\{|X\cup C_{\leq t}|-t : \text{$(C_1,\dots, C_t)$ is a proper ${\cal C}_0$-sequence in $M$}\} \qquad (X\subseteq E)
\]
gives an upper-bound for the rank function of $M$.
 It follows that, if $f_{{\cal C}_0}$ is the rank function of some matroid on $E$,
 then this  matroid will be the unique maximal ${\cal C}_0$-matroid and hence will be the free elevation of $M_0$  by Lemma \ref{lem:unique}.
Since the poset of all elevations of $M_0$ may not have a unique maximal element,
%
%there is no unique maximal  ${\cal C}_0$-matroid in Brylawski's example~\cite[Figure 7.9]{B86},
% gave an example
%showing that the free elevation $M_0^*$ is not the unique maximal
%${\cal C}_0$-matroid on $E$ in general.
%Hence
$f_{{\cal C}_0}$ is not always a matroid rank function.
%(it is not submodular in this example).
We believe that $f_{{\cal C}_0}$ is  a matroid rank function  whenever
%the free elevation of $M_0$ is
there is a unique maximal ${\cal C}_0$-matroid on $E$.

\begin{conjecture}\label{conj:rank}
Let $M_0$ be a matroid on a finite set $E$ and let ${\cal C}_0$ be the family of non-spanning circuits in $M_0$.
Suppose that there is a unique maximal ${\cal C}_0$-matroid on $E$.
Then $f_{{\cal C}_0}$ is the rank function of this maximal ${\cal C}_0$-matroid.
\end{conjecture}
Our main result verifies Conjecture~\ref{conj:rank} when $M_0$ is the rank 10 matroid on $E(K_n)$ in which the set of  non-spanning  circuits ${\cal C}_0$ is the set of copies of $K_5$ in $K_n$. (We already showed in \cite{CJT1} that the cofactor matroid $\mathcal{C}^1_2(K_n)$ is the  
%will see below
%in Section \ref{sec:rigidity} below
%that this 
unique maximal $K_5$-matroid 
%is the maximal abstract 3-rigidity matroid 
on $E(K_n)$ and Theorem \ref{thm:rank} below verifies that  $f_{{\cal C}_0}$ is its rank function.) The conjecture is verified for several other matroids on $E(K_n)$ in \cite{JT}.

\subsection{A covering lemma}
We will use the algorithm for constructing a free elevation  given in Section~\ref{sec:erections} to show that every cyclic flat in the free elevation of a matroid $M_0$ can be covered by the non-spanning circuits of $M_0$ (Lemma~\ref{lem:covering} below). This will be a key tool in proving that Conjecture \ref{conj:rank} holds for the maximal $K_5$-matroid on $E(K_n)$. 
%We first need to establish the following preliminary lemma.

%\begin{lemma}
%\label{lem:modular1}
%Let $M$ be a matroid and $I,J\subseteq E(M)$.
%Suppose that $(I, J)$ is modular.
%Then $(I', J')$ is modular for all $I', J'$ with $I\subseteq I', J\subseteq J', r_M(I)=r_M(I')$ and $r_M(J)=r_M(J')$.
%\end{lemma}
%\begin{proof}
%We have
%\begin{align*}
%r_{M}(I)+r_{M}(J)&=r_{M}(I')+r_{M}(J') & (\text{by $r_M(I)=r_M(I')$ and $r_M(J)=r_M(J')$})\\
%&\geq r_{M}(I'\cup J')+r_{M}(I'\cap J') & (\text{by submodularity}) \\
%&\geq r_{M}(I\cup J)+r_{M}(I\cap J)  & (\text{by $I\subseteq I', J\subseteq J$})\\
%&=r_{M}(I)+r_{M}(J) & (\text{by modularity})
%\end{align*}
%and hence equality must hold in each inequality.
%\end{proof}

Recall that ${\cal CF}_M$ denotes the family of non-spanning cyclic flats in a matroid $M$, and that
$\overline{\cal CF}_M$ denotes its modular cyclic closure.
We will need the following characterization of the cyclic hyperplanes in %result determines the non-spanning cyclic flats of  maximum rank in 
the free erection of $M$.

\begin{lemma}
\label{lem:hyperplane}
Suppose $M$ is a matroid on a finite set $E$, $N$ is the free erection of $M$, and  $N\neq M$. Let $Z\subseteq E$.
%Let
%${\cal CF}$ be the family of all cyclic flats in $N$ other than $E$, and
%$Z$ be a cyclic set in $M$. 
Then
$Z$ is a cyclic hyperplane in $N$ if and only if $Z$ is a spanning set in $M$ and a maximal element in $\overline{\cal CF}_M$.
\end{lemma}
%Note: this is slightly different to the corresponding lemma in \cite{CJT2}. I added the hypothesis that $N\neq M$ (which i think we need) and moved the hypothesis that $Z$ is cyclic to part of the conclusion.
\begin{proof}
Suppose that $Z$ is spanning in $M$
and  a maximal element in $\overline{\cal CF}_M$.
%As $Z$ is spanning in $M$ and $Z\in \overline{\cal CF}$,
Then $Z$ is cyclic in both $M$ and $N$, and
$r_M(Z)=r_N(Z)=r_N(E)-1$ by (\ref{eq:erection}).
It remains  to show that $Z$ is a flat in $N$.
To see this, choose a base $I$ of $M$ with  $I\subseteq Z$.
Then for any $e\notin Z$, $I+e$ contains a circuit $C$ of $M$ with $e\in C$.
As $Z\cup C=Z+e$, $Z+e$ is cyclic in $M$, and the maximality of $Z$ in $\overline{\cal CF}_M$ now gives $Z+e \notin \overline{\cal CF}_M$.
Hence $r_N(Z+e)>r_N(Z)$ for all $e\notin Z$, which means that $Z$ is a flat in $N$.

Conversely, suppose that $Z$ is a cyclic hyperplane in $N$.
Then $Z$ is a non-spanning cyclic set in $N$ so $Z\in \overline{\cal CF}_M$. 
By the definition of truncation,
$Z$ is a spanning set in $M$. 
%Since $r_M(Z)=r_N(Z)$ and $Z$ is cyclic in $M$, we have $Z\in \overline{\cal CF}_M$ by (\ref{eq:erection}).
If $Z$ is not maximal in $\overline{\cal CF}_M$, then there is a maximal set $Y\in \overline{\cal CF}_M$ with $Z\subsetneq Y$.
Then by the first part of the proof, $Y$ would be a hyperplane in $N$, which is a contradiction since $Z\subsetneq Y$.
\end{proof}

Our covering lemma for free elevations will follow by recursively applying our next result at each step in the sequence of free erections used to construct a free elevation.
 
\begin{lemma}\label{lem:coveringnew}
Let $M$ be a matroid on a finite set $E$, ${\cal C}$ be a family of  non-spanning circuits of $M$, and $N$ be the free erection of  $M$.
Suppose that each cyclic flat in $M$ is the union of circuits in ${\cal C}$.
Then each cyclic flat in $N$ is the union of circuits in ${\cal C}$.
\end{lemma}

\begin{proof}
The theorem is trivially true if $M=N$, so we may assume that $M\neq N$. This implies that some spanning circuit of $M$ becomes a base of $N$. Hence $M$ has no coloops,  $E$ is a cyclic flat of $M$ and, by hypothesis, each $e\in E$ belongs to a circuit in $\cal C$. 
%We proceed by induction on $i$.
%We first consider the base case, $i=0$.
%Let $F$ be a cyclic flat in $M_0$.
%If $F$ is not spanning, then it is the union of non-spanning circuits of $M_0$ (as it is cyclic), and hence it is the union of circuits in ${\cal C}_0$.
%If $F$ is spanning, then $F=E$ and the lemma follows from the hypothesis that each element in $E$ is contained in a circuit in ${\cal C}_0$.
%
%Suppose
%%We then prove the statement for $M_{i+1}$, assuming
%that the lemma holds for $M_i$, for some $i\geq 0$, and let ${\cal CF}_i$ be the collection of all non-spanning cyclic flats in $M_i$.
We first prove the following.

\begin{claim}
\label{claim:cover}
For each $Z\in \overline{\cal CF}_M$, there is a $Z'\in \overline{\cal CF}_M$ such that $Z\subseteq Z'$, $r_{M}(Z)=r_{M}(Z')$, and $Z'$ is the union of circuits in  ${\cal C}$.
\end{claim}

\begin{proof}
Let ${\cal S}_0, \dots, {\cal S}_k$ be the families of cyclic sets defined in the construction of $\overline{\cal CF}_M$ from  ${\cal CF}_M$ in Algorithm 1. 
%see Section~\ref{sec:erections}, 
%and let $S_j^{\downarrow}$ be the down-closure of $S_j$ in the lattice of cyclic sets of $M_i$.
Since $Z\in \overline{\cal CF}_M$, $Z\in {\cal S}_j^{\downarrow}$ for some $0\leq j\leq k$. We prove that the claim holds for $Z$ by induction on $j$.

For the  base case, we assume that $Z\in {\cal S}_0^{\downarrow}$.
If $Z\in {\cal S}_0={\cal CF}_M$, then
%so the maximal elements of ${\cal S}_0^{\downarrow}$ are
$Z$ is a cyclic flat of $M$ and hence
%Each cyclic flat
is the union of circuits in ${\cal C}$ by our hypothesis on $M$. On the other hand, if $Z\in {\cal S}_0^{\downarrow}\setminus {\cal S}_0$, then 
$\cl_{M}(Z)\in {\cal S}_0$ 
%is a non-spanning cyclic flat of $M_i$, 
and $Z'=\cl_{M}(Z)$ is the desired set for $Z$.

Suppose $Z\in {\cal S}_j^{\downarrow}\sm {\cal S}_{j-1}^{\downarrow}$ for some $j\geq 1$. 
We first consider the case when $Z\in {\cal S}_j$.
Then, by Algorithm 1, $Z=A\cup B$ for some $A, B\in {\cal S}_{j-1}^{\downarrow}$ which form a modular pair in $M$.
By induction on $j$, there exist $A', B'\in \overline{\cal CF}_M$ with $A\subseteq A'$, $B\subseteq B'$,
$r_{M}(A)=r_{M}(A')$ and
$r_{M}(B)=r_{M}(B')$, such that $A'$ and $B'$ are both the union of sets in  ${\cal C}$.
We claim that $A'$ and $B'$ form a modular pair in $M$. This follows since
\begin{align*}
r_{M}(A)+r_{M}(B)&=r_{M}(A')+r_{M}(B') & (\text{by $r_M(A)=r_M(A')$ and $r_M(B)=r_M(B')$})\\
&\geq r_{M}(A'\cup B')+r_{M}(A'\cap B') & (\text{by submodularity}) \\
&\geq r_{M}(A\cup B)+r_{M}(A\cap B)  & (\text{by $A\subseteq A', B\subseteq B'$})\\
&=r_{M}(A)+r_{M}(B) & (\text{by modularity})
\end{align*}
%
%\begin{align*}
%r_{M}(A)+r_{M}(B)&=r_{M}(A')+r_{M}(B') \\
%&\geq r_{M}(A'\cup B')+r_{M}(A'\cap B') \\
%&\geq r_{M}(A\cup B)+r_{M}(A\cap B) \\
%&=r_{M}(A)+r_{M}(B),
%\end{align*}
and hence equality must hold in each inequality.
Since $\overline{\cal CF}_M$ is closed under the union of modular pairs, this gives $A'\cup B'\in \overline{\cal CF}_M$.
The fact that equality holds throughout the above inequality also implies that $r_{M}(A'\cup B')+r_{M}(A'\cap B')= r_{M}(A\cup B)+r_{M}(A\cap B)$.
This and the monotonicity of $r_{M}$ imply that $r_{M}(A'\cup B')=r_{M}(A\cup B)=r_{M}(Z)$.
Thus $Z'=A'\cup B'$ is the desired set for $Z$.
%, which solves the case when $Z\in {\cal S}_j$.

It remains to consider the case when $Z\in {\cal S}_j^{\downarrow}\setminus ({\cal S}_j\cup {\cal S}_{j-1}^{\downarrow})$. Then ${\rm cl}_M(Z)\not\in {\cal CF}_M$ so 
%If $Z$ is not spanning in $M$, then ${\rm cl}_{M}(Z)$ is a non-spanning cyclic flat of $M$, and hence is a union of sets in ${\cal C}_0$ since the theorem holds for $M$. Thus $Z'={\rm cl}_{M}(Z)$ is the desired set for $Z$.
%Hence we may assume that  
$Z$ is a spanning set in $M$. 
Choose a maximal element $\tilde Z$ of ${\cal S}_j^{\downarrow}$ with $Z\subseteq \tilde Z$.
Then  $\tilde Z\in {\cal S}_j$, and hence there is a set $\tilde Z'$ for $\tilde Z$ by the preceding paragraph.
%case when $Z\in {\cal S}_j$.
Then $r_M(\tilde Z')=r_M(Z)$ since $Z$ is spanning in $M$ and $Z'=\tilde Z'$ is the desired set for $Z$.
This completes the proof of the claim.
\end{proof}

The claim implies,  in particular,  that every maximal element in $\overline{\cal CF}_M$ is the union of circuits in ${\cal C}$.

We can now  complete the proof of the lemma.
Choose a cyclic flat $Z$ of $N$.
If $Z$ is spanning in $N$, then $Z=E$ and the lemma follows  since each element in $E$ is contained in a circuit in ${\cal C}$.
So we may assume that $Z$ is not spanning in $N$.
Then $N|_Z = M|_Z$ follows as $M$ is the truncation of $N$.
As $Z$ is cyclic in $N$, $Z$ is also cyclic in $M$.
If $r_{N}(Z)<r_{M}(E)$, then $Z$ is a non-spanning cyclic flat in $M$ and the claim follows from the hypothesis on $M$. So we may further assume that $r_{N}(Z)=r_{M}(E)=r_{N}(E)-1$.
Then $Z$  is a hyperplane in $N$.
Lemma~\ref{lem:hyperplane} now implies that $Z$ is a maximal element in $\overline{\cal CF}_M$, and hence 
 $Z$ is the union of circuits in ${\cal C}$ by  Claim~\ref{claim:cover}.
This completes the proof of the lemma.
\end{proof}

Lemma \ref{lem:coveringnew} and the fact that non-spanning circuits are preserved when we construct a free erection immediately give:

\begin{lemma}\label{lem:covering}
Let $M$ be a matroid on a finite set $E$, ${\cal C}$ be a family of  non-spanning circuits of $M$, and $N$ be the free elevation of  $M$.
Suppose that each cyclic flat in $M$ is the union of circuits in ${\cal C}$.
Then each cyclic flat in $N$ is the union of circuits in ${\cal C}$.
\end{lemma}

\section{\boldmath $K_{d+2}$-matroids and abstract $d$-rigidity}\label{sec:rigidity}
We will apply the preceding theory to 
%our main concern,
 %now specialize our main target.
%We shall define two classes of matroids, $K_5$-matroids and
abstract $d$-rigidity matroids.
We will see that abstract $d$-rigidity matroids are {\em $K_{d+2}$-matroids}, i.e.~matroids on $E(K_n)$ in which the edge set of every copy of $K_{d+2}$ is a circuit,  and show that they have maximum possible rank over all such matroids. We will then  give two important examples of abstract $d$-rigidity matroids,
%There are two important examples of abstract 3-rigidity matroids,
generic $d$-dimensional rigidity matroids and $C_{d-1}^{d-2}$-cofactor matroids, and describe some related results and conjectures.
%In our first paper, we proved that the $C_2^1$-cofactor matroid is the unique maximal abstract rigidity matroid (and the unique maximal $K_5$-matroid).
%We shall introduce those matroids one by one, and explain their properties.

Let $G=(V,E)$ be a graph. 
For $X\subseteq V$, let $G[X]$ be the subgraph of $G$ induced by $X$.
For $F\subseteq E$, let $V(F)$ be the set of vertices incident to $F$, and let $G[F]=(V(F),F)$.
For $v\in V$, let $N_G(v)$ be the set of neighbors of $v$ in $G$, and let $d_G(v)=|N_G(v)|$.
For $X=\{v_1,\dots, v_k\}\subseteq V$, let $K(X)$ or $K(v_1,\dots, v_k)$ be the edge set of the complete graph on $X$.
%We will often abuse notation and use the same symbol for a subgraph of $G$ and its edge set (considered as a subset of the ground set of a matroid on the edge set of $G$).

Given a set $F$ in a matroid $M$ defined on $E(K_n)$, we will often abuse notation and use the same letter for both the set $F$ and the graph 
%of $K_n$, 
$K_n[F]$.
%  induced by $F$. 
It will be clear from the context whether we are referring to an edge set or a graph. In addition, we will refer to bases of $M|_F$ as {\em bases of $F$}.

%\subsection{\boldmath $K_{d+2}$-matroids}
%%Let ${\cal M}$ be a matroid on the edge set of the complete graph $K_n$.
%%We say that ${\cal M}$ is
%A {\em $K_{d+2}$-matroid} is a matroid ${\cal M}$ on the edge set of the complete graph $K_n$ in which the
%edge set of every copy of $K_{d+2}$  is a circuit.  (The reason for the subscript `$d+2$' will soon become apparent.)
We first use Lemma~\ref{lem:rank1} to obtain an upper bound on the rank of any $K_{d+2}$-matroid. We say that a sequence $(C_1,\dots, C_t)$ of edge sets in $K_n$ is a {\em $K_{t}$-sequence} if each $C_i$ induces a copy of $K_{t}$ in $K_n$.

\begin{lemma}\label{lem:upper_bound}
Let $M$ be a $K_{d+2}$-matroid on the edge set of the complete graph $K_n$ with $n\geq d+2$.
Then its rank is at most $dn-{d+1 \choose 2}$.
\end{lemma}
\begin{proof} It will suffice to construct a proper $K_{d+2}$-sequence which covers $K_n$ of length ${{n-d}\choose{2}}$ since Lemma \ref{lem:rank1} will then give $r(M)\leq {{n}\choose{2}}-{{n-d}\choose{2}}=dn-{{d+1}\choose{2}}$. To this end, we let $K_n=K(v_1,v_2,\ldots,v_n)$ and construct a proper $K_{d+2}$-sequence ${\cal K}_i$ of length ${{i-d}\choose{2}}$ which covers $K(v_1,v_2,\ldots, v_i)$ for all $d+2\leq i\leq n$, recursively.
We first put ${\cal K}_{d+2}=(K(v_1,v_2,\ldots,v_{d+2}))$. Then, for each $d+2\leq i\leq n-1$, we construct  a proper $K_{d+2}$-sequence ${\cal K}_{i+1}$ which covers $K(v_1,\dots, v_{i+1})$ by choosing a set $S_{i}$ of $d$ vertices in $K(v_1,v_2,\ldots, v_i)$, putting ${\cal L}_{i+1}=(K(S_i+v_j+v_{i+1})\,:\,1\leq j\leq i \mbox{ and }j\not\in S_i)$, and then putting
${\cal K}_{i+1}=({\cal K}_{i},{\cal L}_{i+1})$.
It is straightforward to check that ${\cal K}_{i+1}$ has length ${{i+1-d}\choose{2}}$ and is proper (for any ordering of the $K_{d+2}$'s in each subsequence ${\cal L}_j$).
\end{proof}

The graphic matroid of $K_n$ and the rank-two uniform matroid on $E(K_n)$ are examples of $K_3$-matroids. In addition, the graphic matroid achieves the upper bound on the rank given by  Lemma \ref{lem:upper_bound}.

%\subsection{\boldmath Abstract rigidity matroids}\label{subsec:abstract_rigidity}
%Let ${\cal M}$ be a matroid on the edge set of the complete graph $K_n$.
%${\cal M}$ is an
%{\em Abstract $d$-rigidity matroids} were introduced by 
Graver~\cite{G91} defined an {\em abstract $d$-rigidity matroid} to be
a matroid $M$ on the edge set of the complete graph $K_n$ which satisfies the following two axioms:
\begin{description}
\item[(A1)] If $E_1, E_2\subseteq E(K_n)$ with $|V(E_1)\cap V(E_2)|\leq d-1$, then
${\rm cl}_M(E_1\cup E_2)\subseteq K(V(E_1))\cup K(V(E_2))$;
\item[(A2)] If $E_1, E_2\subseteq E(K_n)$ with ${\rm cl}_M(E_1)=K(V(E_1))$, ${\rm cl}_M(E_2)=K(V(E_2))$, and $|V(E_1)\cap V(E_2)|\geq d$, then
${\rm cl}_M(E_1\cup E_2)= K(V(E_1\cup E_2))$.
\end{description}
These two conditions reflect two fundamental rigidity properties of generic $d$-dimensional bar-joint frameworks. (See, for example, \cite{GSS93} for more details.)
%A long-standing conjecture (Conjecture~\ref{conj:max}) claims that
%${\cal R}_{3,n}$ is the unique maximal abstract 3-rigidity matroid, and hence
%understanding a maximal $d$-rigidity matroid would be an important step  toward combinatorial characterization of 3-rigidity.
Nguyen~\cite{N10} obtained a simple characterization of abstract $d$-rigidity matroids which immediately implies that they are special kinds of $K_{d+2}$-matroids.

\begin{theorem}[Nguyen{\cite[Theorem 2.2]{N10}}]\label{thm:hang}
Let $n,d$ be positive integers with $n\geq d+2$ and $M$ be a matroid on
%the edge set of the complete graph
$E(K_n)$.
Then $M$ is an abstract $d$-rigidity matroid if and only if
$M$ is a $K_{d+2}$-matroid with  rank $dn-{d+1\choose 2}$.
\end{theorem}

Lemma~\ref{lem:upper_bound} and Theorem~\ref{thm:hang} imply that abstract $d$-rigidity matroids are precisely the $K_{d+2}$-matroids on $E(K_n)$ which attain the maximum possible rank.

\subsection{Two fundamental examples} \label{subsec:examples}
We describe two examples of abstract $d$-rigidity matroids which have been extensively studied in the literature.

\paragraph{\boldmath Generic rigidity matroids.}%\label{subsec:generic_rigidity}
A {\em $d$-dimensional framework} is a pair $(G,p)$ consisting of  a graph $G=(V,E)$ and a map $p:V\rightarrow \mathbb{R}^d$.
An {\em infinitesimal motion} of $(G,p)$ is a map $q:V\rightarrow \mathbb{R}^d$ such that
\begin{equation}\label{eq:inf}
(p(u)-p(v))\cdot (q(u)-q(v))=0 \qquad (uv\in E),
\end{equation}
where $\cdot$ denotes the Euclidean inner product.
The {\em rigidity matrix} $R(G,p)$ of $(G,p)$ is the matrix representing the  linear system  of equations (\ref{eq:inf}) in the variables $q$.
Specifically, $R(G,p)$ is a matrix of size $|E|\times d|V|$ in which each row is indexed by an edge, sets of $d$ consecutive columns are indexed by the vertices, and the row indexed by the  edge $e=uv$ has the form:
\[
\kbordermatrix{
  & &  u & & v & \\
 e=uv & 0 \cdots 0 & p(u)-p(v) & 0\cdots 0 & p(v)-p(u) & 0\cdots 0
}.
\]
The framework $(G,p)$ is said to be {\em infinitesimally rigid} if $\rank R(G,p)=dn-{d+1\choose 2}$.

The {\em rigidity matroid} of $(G,p)$ is the matroid on $E$ in which a set of edges is independent if the corresponding rows of $R(G,p)$ are linearly independent.
%Since the set of rows of $R(G,p)$ is indexed by $E$, the row vectors of $R(G,p)$ define a matroid on $E$, which is called the {\em rigidity matroid} of $(G,p)$.
The rigidity matroid of $(G,p)$ is the same for any generic
$p:V\to \R^d$ and we refer to this matroid as
%Although the rigidity matroid depends on $p$, the maximal rigidity matroid of $(G,p)$ is obtained if we take a(ny) generic $p$,
%and it is determined by $G$ and $d$.
%The maximal rigidity matroid is called
the {\em generic $d$-dimensional rigidity matroid} ${\cal R}_d(G)$ of $G$.
%In particular,
The generic $d$-dimensional rigidity matroid of the complete graph $K_n$ is called the {\em generic $d$-dimensional rigidity matroid} (on $n$ vertices) and is denoted by ${\cal R}_{d,n}$.
A graph $G=(V,E)$ with $n$ vertices is said to be {\em rigid in $\R^d$} if $E$ is a spanning set of ${\cal R}_{d,n}$
or equivalently, when $n\geq d$, if the rank of $E$ is $dn-{d+1\choose 2}$.
Combinatorial characterizations of graphs which are rigid in $\R^d$ exist for $d=1,2$. Obtaining an analogous characterization for $d\geq 3$ is
the central open problem in graph rigidity.
 %the central open problem in graph rigidity theory.
%Further details will be discussed in Section~\ref{subsec:inductive_rigidity}.

\paragraph{Generic cofactor matroids.}%\label{subsec:generic_cofactor}
Let $G=(V,E)$ be a graph and $p:V\to \R^2$ such that $p(v_i)=(x_i,y_i)$ for all $v_i\in V$.
%and $p(v_i)\neq p(v_j)$ for all $v_iv_j\in E$.
We assume that  the vertices are ordered as $v_1,v_2,\ldots,v_n$.
%Following Whiteley \cite{Wsurvey}, 
As given in the introduction,
the {\em $C^{s-1}_s$-cofactor matrix} of $(G,p)$, denoted by $C_s^{s-1}(G,p)$, is a matrix of size
$|E|\times (s+1)|V|$  in which each row is indexed by an edge, sets of $(s+1)$ consecutive columns are indexed by the vertices,
%each set of consecutive $(s+1)$ columns is associated with a vertex, each row is associated with an edge,
and the row associated to the  edge $e=v_iv_j$ with $i<j$ is
\[
\kbordermatrix{
 & & v_i & & v_j & \\
 e=v_iv_j & 0\cdots 0 & D_{ij} & 0 \cdots 0 & -D_{ij} & 0\cdots 0
},
\]
where  $D_{i,j}=((x_i-x_j)^s, (x_i-x_j)^{s-1}(y_i-y_j),\dots, (x_i-x_j)(y_i-y_j)^{s-1}, (y_i-y_j)^s)\in \mathbb{R}^{s+1}$.
(Our definition is slightly different to that given by Whiteley \cite{Wsurvey}, but the two definitions are equivalent up to elementary column operations.)
%It is known that
When $s\geq 1$, the space $S_s^{s-1}(\Delta)$ of bivariate $C_s^{s-1}$-splines over $\Delta$ is linearly isomorphic to the left kernel of  $C_s^{s-1}(G,p)$ if $(G,p)$ is the 1-skeleton of a subdivision $\Delta$ of a polygonal domain in the plane,
see, e.g., \cite{Wsurvey} for more details. Note that, when $s=0$, $D_{i,j}\equiv 1$ and $C_s^{s-1}(G,p)$ is just the edge/vertex incidence matrix of $G$. 

The {\em generic $C_s^{s-1}$-cofactor matroid} ${\cal C}_{s,n}^{s-1}$ is defined to be the row matroid of $C_s^{s-1}(K_n,p)$ for any generic $p$. Whiteley~\cite{Wsurvey} showed that  ${\cal C}_{d-1,n}^{d-2}$ is an abstract $d$-rigidity matroid for all $d\geq 1$. It follows immediately from the definitions of the rigidity and cofactor matrices  that ${\cal C}_{d-1,n}^{d-2}={\cal R}_{d,n}$ when $d=1,2$.  Whiteley~\cite{Wsurvey} showed that that ${\cal C}_{d-1,n}^{d-2}\neq{\cal R}_{d,n}$ when $d\geq 4$ and $n\geq 2d+2$ and conjectured that the two matroids are equal when $d=3$.

\subsection{Maximality conjectures}
As noted above, abstract $d$-rigidity matroids are $K_{d+2}$-matroids which attain the maximum possible rank.
Graver~\cite{G91} conjectured further that the generic $d$-dimensional rigidity matroid is the unique maximal matroid in the weak order poset of all abstract $d$-rigidity matroids on $E(K_n)$, and verified his conjecture when $d=1,2$.
%The conjecture is known to be true in $d=1,2$, which means that
%the graphic matroid is the unique maximal matroids among $K_3$-matroids and
%the generic 2-rigidity matroid is the unique maximal matroids among $K_4$-matroids.
 N.~J.~Thurston (see, \cite[page 150]{GSS93}) subsequently showed that Graver's conjecture is false when $d\geq 4$.
%On the other hand, the conjecture for $d\geq 4$ was disproved by ?.
The conjecture remains as a long-standing open problem for $d=3$.

\begin{conjecture}[Graver's maximality conjecture~{\cite{G91}}]\label{conj:rigidty_maximality}
The generic $3$-dimensional rigidity matroid is the unique maximal abstract $3$-rigidity matroid on $E(K_n)$.
\end{conjecture}

 Whiteley~\cite{Wsurvey}  developed a theory for $C_{d-1}^{d-2}$-cofactor matroids which is analogous to that for $d$-dimensional rigidity matroids. In particular, he provided further counterexamples to Graver's maximality conjecture  by noting that the edge set of the complete bipartite graph $K_{d+2,d+2}$ is dependent in ${\cal R}_{d,n}$ for all $d\geq 1$ but independent in $C_{d-1,n}^{d-2}$ when $d\geq 4$.  This led him to make the following modified conjecture.

\begin{conjecture}[Whiteley's maximality conjecture~{\cite[Conjecture 11.5.1]{Wsurvey}}]\label{conj:max_cofactor}
The generic  $C_{d-1}^{d-2}$-cofactor matroid ${\cal C}_{d-1,n}^{d-2}$ is the unique maximal abstract $d$-rigidity matroid on $E(K_n)$ for all $d\geq 1$.
\end{conjecture}

The conjecture holds for
% $d=2$, and is also valid for 
$d=1,2$, since Graver's conjecture holds for $d=1,2$ and the generic cofactor and rigidity matroids are the same for these values of $d$.
%in these case, tby observing that $C_1^0(G,p)=R(G,p)$ for any 2-dimensional framework $(G,p)$.
The main result of our first paper in this series \cite{CJT1} verifies Whiteley's maximality conjecture when $d=3$.
 \begin{theorem}\label{thm:cofactor}
  The generic $C_{2}^1$-cofactor matroid ${\cal C}^1_{2,n}$ is the unique maximal $K_5$-matroid on $E(K_n)$.
  \end{theorem}
 Theorem \ref{thm:cofactor} is stronger than Whiteley's conjecture since it holds for the larger class of $K_5$-matroids.
We can use a similar proof technique to that given by Graver to show that analogous results hold when $d=1,2$:
the generic 1-dimensional rigidity matroid is the unique maximal $K_3$-matroid and   the generic 2-dimensional rigidity matroid
is the unique maximal $K_4$-matroid on $E(K_n)$.
Whiteley's conjecture, and the corresponding strengthening to the larger class of $K_{d+2}$-matroids,   remain open for all $d\geq 4$.

\subsection{Inductive constructions}\label{subsec:inductive_rigidity}
Inductive constructions are  frequently used to solve problems in rigidity.
The most common operations in this context are $k$-extensions, which are defined as follows.
Given integers $k$ and $d$ with $0\leq k\leq d$, a {\em
 $d$-dimensional  $k$-extension} of a graph $G$ constructs a new graph by deleting $k$ edges from $G$ and then adding a new vertex $v_0$ and $d+k$ new edges incident to $v_0$, with the proviso that each end-vertex of a deleted edge is a neighbor of $v_0$. 
% 
% 
% pinching\footnote{More formally,  a $k$-extension removes $k$ existing edges $u_1v_1, \dots, u_kv_k$ and adds a new vertex $v_0$ and new $2k+(d-k)$ edges   $v_0u_1, \dots, v_0u_k, v_0v_1, \dots, v_0v_k, v_0w_1,\dots, v_0w_{d-k}$ with $\{w_1,\dots, w_{d-k}\}\cap \{u_1,\dots, u_k, v_1,\dots, v_k\}=\emptyset$.} $2k$ existing edges at a new vertex $v_0$ and then adding $(d-k)$ new edges at $v_0$, so that no parallel edges occur.
 See Figure~\ref{fig:X} for an example.

\begin{figure}[t]
\centering
\includegraphics[scale=1]{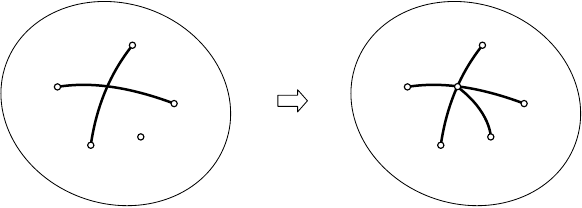}
\caption{An example of 3-dimensional 2-extension (X-replacement).}
\label{fig:X}
\end{figure}

It is well-known that,
% the $d$-dimensional $k$-extension operation preserves rigidity in $\R^d$ when $k=0, 1$, i.e.,
if $G$ is rigid in $\R^d$ and $k=0,1$, then any graph obtained from $G$ by a $k$-extension is rigid  in $\R^d$, see for example \cite{Wsurvey}.
This fact was used by Pollaczek-Geiringer \cite{P-G} to show that Maxwell's necessary condition for rigidity (given in the introduction) is also sufficient in dimension two.
%
%can be used to prove that Maxwell's condition (given in the introduction) is also sufficient in dimension two~\cite{P-G,Lam}
%since any graph satisfying the $2$-dimensional Maxwell condition  can be built from a triangle by a sequence of 2-dimensional $k$-extensions with $k=0,1$ (see, for example, \cite{TW85}).
% as any graph satisfying the $2$-dimensional Maxwell condition (given in the introduction) can be built from a triangle by a sequence of 2-dimensional $k$-extensions with $k=0,1$ (see, for example, \cite{TW85}).
%The  is widely known as {\em Henneberg construction}.

 The situation becomes more complicated when $k\geq 2$.
In particular, we can distinguish two types of $2$-extensions, depending on whether or not the two deleted edges are adjacent:
a $2$-extension is called a {\em V-replacement} if the deleted edges are adjacent,
and otherwise is called an {\em X-replacement}. See Figure~\ref{fig:X}.
It is conjectured that 3-dimensional X-replacement preserves rigidity in $\R^3$.

\begin{conjecture}[X-replacement conjecture~\cite{TW85}]
%Let $H$ be a graph, $v$ be a vertex with degree five, and $G$ be a graph obtained from $H$ by removing $v$,
%and adding two non-adjacent edges on $N_H(v)$.
Suppose that $G$ and $H$ are graphs and that $H$ can be obtained from $G$ by a $3$-dimensional X-replacement.
If $G$ is rigid in $\R^3$, then $H$ is rigid in $\R^3$.
\end{conjecture}
\noindent
See \cite{C14} for more details about this conjecture.
Maehara~\cite{M91} pointed out that the  analogous statement for 4-dimensional $X$-replacement does not hold in general.

The $d$-dimensional $k$-extension operation can be applied to any graph and hence can be used to investigate independence properties of any matroid $M$ defined on the edge set of $K_n$.
%
%
%On the other hand, Whiteley~\cite[Theorem 11.4.1]{Wsurvey} showed that $d$-dimensional $X$-replacent preserves  `$C_{d-1}^{d-2}$-rigidity' for all $d$.
%This is a major reason why we were able to solve his maximality conjecture for $C_2^1$-cofactor matroids.
%
%%
%%For V-replacement, there is an easy example of showing that $V$-replacement does not work in any $d$ (see, e.g., \cite{TW85}), but the following conjecture is still open.
%%\begin{conjecture}[Double V-replacement conjecture~\cite{TW85}]
%%Let $H$ be a graph, $v$ be a vertex with degree five,
%%and $G_i$ be a graph obtained from $H$ by removing $v$,
%%and adding two edges $u_iv_i$ and $v_iw_i$ on $N_H(v)$ for $i=1,2$ with $v_1\neq v_2$.
%%If $G_1$ and $G_2$ are $3$-rigid, then $H$ is $3$-rigid.
%%\end{conjecture}
%
%We can abstract these construction properties for graphs as those for matroids.
%Let ${\cal M}$ be a $K_{d+2}$-matroid.
We say that $M$ has the {\em $d$-dimensional $k$-extension property} if
every edge set obtained from an independent set by a $d$-dimensional $k$-extension operation remains independent.
%By the same manner the $X$-replacement property is defined.

A basic fact about abstract $d$-rigidity matroids is that any abstract $d$-rigidity matroid has the $d$-dimensional 0-extension property~\cite{GSS93}.
%A classical fact known as Henneberg constructions claims
As noted above, the generic $d$-dimensional rigidity matroid also has the $d$-dimensional 1-extension property.
% for all $d\geq 1$.
For $C_{d-1}^{d-2}$-cofactor matroids, Whiteley \cite[Theorem 11.4.1]{Wsurvey} proved the following:
\begin{theorem}
\label{thm:construction_cofactor}
The generic $C_{d-1}^{d-2}$-cofactor matroid has the $d$-dimensional $1$-extension property for all $d\geq 1$, and the $d$-dimensional X-replacement property for all $d\geq 2$.
% for all $d\geq 2$.
\end{theorem}

Theorem \ref{thm:construction_cofactor} is an important ingredient  in our proof of Theorem \ref{thm:cofactor}. The fact that the generic $3$-dimensional rigidity matroid  is not known to have the $3$-dimensional X-replacement property is a major barrier to applying the same proof technique to solve Graver's Maximality Conjecture.
%the generic $3$-dimensional rigidity matroid.

\section{Combinatorial characterization}\label{sec:char}
We saw in the last section that ${\cal C}_{2,n}^1$ is the unique maximal abstract 3-rigidity matroid on $E(K_n)$.
In this section, we obtain a 
%good 
co-NP type characterization of its rank function.
Since ${\cal C}_{2,n}^1$ is the free matroid on $E(K_n)$ when $n\leq 4$, we assume throughout this section that $n\geq 5$.

Theorem \ref{thm:hang} implies that the rank 10 truncation of any abstract 3-rigidity matroid on $E(K_n)$ is  the  matroid $R_{n}$ on $E(K_n)$ of rank $10$ in which the non-spanning circuits are the edge sets of the copies of $K_5$ in $K_n$. We can use this observation, together with other 
results from  Sections~\ref{sec:c-matroid} and  \ref{sec:rigidity} to derive two fundamental properties of ${\cal C}^1_{2,n}$.

\begin{theorem}\label{thm:K5free}
%Let ${R}_0$ be the rank $10$ matroid on $E(K_n)$ in which the non-spanning circuits are the edge sets of the copies of $K_5$ in $K_n$
%. obtained from the generic 3-rigidity matroid ${\cal R}_{3,n}$ by truncating to rank 10,
%and let 
%Let $R$ be the free elevation of ${R}_0$.
%Then ${R}$ is the unique maximal $K_5$-matroid on $E(K_n)$ and is equal to ${\cal C}_{2,n}^1$.
The generic $C_2^1$-cofactor matroid ${\cal C}_{2,n}^1$ is the free elevation of $R_{n}$.
\end{theorem}
\begin{proof}
%Note that  an edge set is a non-spanning circuit in ${\cal R}^{\downarrow}$ if and only if it is the edge set of a copy of $K_5$.
%Hence ${\cal R}^{\downarrow}$
%Since $R_0$ is a $K_5$-matroid,
Theorem~\ref{thm:cofactor} tells us that the generic cofactor matroid ${\cal C}^1_{2,n}$ is the unique maximal $K_5$-matroid on $E(K_n)$.
Lemma~\ref{lem:unique} now implies that
%the three matroids,
${\cal C}^1_{2,n}$ is the free elevation of $R_{n}$.
%, and $R$ is
%the unique maximal $K_5$-matroid.
%, and the free elevation ${\cal R}^*$, coincide.
\end{proof}

%we have the following:
\begin{corollary}\label{cor:maximal}
Every cyclic flat of ${\cal C}^1_{2,n}$ is the union of copies of $K_5$.
%The maximal abstract 3-rigidity matroid has the following properties:
%\begin{itemize}
%\item the $1$-extension property;
%\item the $X$-replacement property;
%\item every cyclic flat is the union of copies of $K_5$.
%\end{itemize}
\end{corollary}
\begin{proof}
%The first two properties follow from Theorems~\ref{thm:construction_cofactor} and \ref{thm:K5free}. The third property 
This follows from Theorem \ref{thm:K5free},
 Lemma~\ref{lem:covering}, and the facts that each non-spanning cyclic flat of $R_n$ is a copy of $K_5$ and  $E(K_n)$ can be covered by copies of $K_5$.
\end{proof}

\subsection{Matroids  on the edge set of a graph}
Our derivation of the rank function of ${\cal C}^1_{2,n}$ uses the property that each of its cyclic flats  has a base which induces a subgraph of $K_n$ of minimum degree at most four.  We will deduce this property from the more general results given in this subsection.
%%3-
%$d$-rigidity matroids which is given in Corollary \ref{cor:cyclicdegree} below.
We first need to introduce the matroidal concepts of connectivity and ear decomposition.

A matroid ${M}$ with ground set $E$ is said to be {\em connected} if, for every pair $e, f\in E$ with $e\neq f$, $M$ has a circuit containing both $e$ and $f$.
A set $X\subseteq E$ is said to be {\em connected} if $M|_X$ is connected. We can define an equivalence relation on $X$ by saying that $e_1,e_2\in X$ are related if either $e_1=e_2$ or there is a circuit of $M$ in $X$ which contains both $e_1$ and $e_2$.
The corresponding equivalence classes $X_1,X_2\ldots, X_k$ are the maximal connected subsets of $X$. We will refer to these equivalence classes as the {\em connected components} of $X$ in $M$. We have $X_i\cap X_j=\emptyset$ for all $1\leq i<j\leq k$ and $r(X)=\sum_{i=1}^k r(X_i)$.

A {\em partial ear decomposition} of   ${M}$ is a sequence $(C_1, C_2\dots, C_t)$ of circuits of $M$ such that, for all $2\leq i\leq t$:
\begin{itemize}
\item $C_i\cap C_{\leq i-1}\neq \emptyset$;
%, where $D_i=C_1\cup \dots \cup C_i$;
\item $C_i\setminus C_{\leq i-1}\neq \emptyset$;
\item if some circuit $C$ of $M$ satisfies $C\cap C_{\leq i-1}\neq \emptyset\neq C\setminus C_{\leq i-1}$ then
$C\setminus C_{\leq i-1}$ is not a proper subset of  $C_i\setminus C_{\leq i-1}$.
\end{itemize}
An {\em ear decomposition} of  $M$ is a partial ear decomposition with the additional property  that $C_{\leq t}=E$. More generally,
an {\em ear decomposition} of a set $X\subseteq E$ is a partial ear decomposition with the additional property  that $C_{\leq t}=X$. Ear decompositions of matroids were introduced by Coullard and Hellerstein \cite{CH}. They showed:

\begin{theorem}\label{thm:ear}
A  matroid $M$ with at least two elements is connected if and only if it has an ear decomposition. Furthermore, if
$M$ is connected, then any partial ear decomposition of $M$ can be extended to an ear decomposition of $M$.
\end{theorem}

We will use an ear decomposition of a connected set $X$ in a matroid $M$ defined on the edge set of a graph $G$ to show that, subject to a mild condition on the circuits of $M$, some base of $X$ induces a subgraph of $G$ of low minimum degree.  More precisely, we use induction on the number of circuits in the ear decomposition
% of $M|_X$ 
to show that the average minimum degree over all bases of $X$ is at most $2\frac{r(X)+1}{|V(X)|}-1$. For the inductive step we will need the following result which follows easily from the definition of an ear decomposition.  %SHOULD WE GIVE A PROOF FOR COMPLETENESS? BILL

\begin{lemma}\label{lem:ear} Let $(C_1,C_2,\ldots,C_t)$ be a partial  ear decomposition in a  matroid $M
$. Then $r(C_{\leq i})=r(C_{\leq i-1})+|C_i\setminus C_{\leq i-1}|-1$ for all $2\leq i\leq t$.
Furthermore, if $B_{i-1}$ is any base of $C_{\leq i-1}$ and $Y$ is any subset of $C_i\setminus C_{\leq i-1}$ of size $|C_i\setminus C_{\leq i-1}|-1$ then $B_i=B_{i-1}\cup Y$ is a base of $C_{\leq i}$.
\end{lemma}

\begin{lemma}\label{lem:degree}
Let $M$ be a matroid defined on the edge set of a graph $G=(V,E)$ and $X\subseteq E$ be a connected set  in  $M$. 
%with $|X|\geq 2$. 
Suppose that every circuit of $M$ induces a $2$-connected subgraph of $G$.  Then
%$K_5$-matroid
$$
f(X):=\sum_{v\in V(X)}\min\{d_B(v):\mbox{$B$ is a base of $X$}\}\leq 2(r(X)+1)-|V(X)|.$$
%In particular, $X$ has a base $B$ whose minimum degree is at most four.
\end{lemma}
\begin{proof}
We use induction on $|X|$. The lemma holds when $X$ is a coloop since we have $r(X)=|X|=1$ and $|V(X)|\leq 2$. Hence we may suppose that $|X|\geq 2$ so $X$ is a cyclic set in $M$.  Let $n=|V(X)|$.

%Let $f(X)=\sum_{v\in V(X)}\min\{d_B(v):\mbox{$B$ is a basis of $X$}\}$. 
We first consider the case when $X$ is a circuit of $M$. Then $r(X)=|X|-1$. For each $v\in V(X)$, we can construct a base of $X$ which contains $d_X(v)-1$ edges incident to $v$, so the contribution of $v$ to $f(X)$ is $d_X(v)-1$. Hence
$$f(X)=\sum_{v\in V(X)}(d_X(v)-1)=2|X|-n=2(r(X)+1)-n
$$
as required. Thus we may assume that $X$ is not a circuit.

Let $(C_1,C_2,\ldots,C_m)$ be an ear decomposition of $X$ and put $Z=C_m\sm C_{\leq m-1}$. By Lemma~\ref{lem:ear}, $r(X)=r(C_{\leq m-1})+|Z|-1$ holds and, for  every base $B'$ of $C_{\leq m-1}$ and every  subset $Y$ of $Z$ of cardinality $|Z|-1$, $B=B'\cup Y$ is a base of $X$.

We will obtain an upper bound on $f(X)$ by considering the separate contributions of the vertices in  $V_1=V( C_{\leq m-1})$ and $V_2=V(X)\sm V_1$. To this end, we let $n_1=|V_1|$, $n_2=|V_2|$, $Z_1$ and $Z_2$ be the sets of edges in $Z$ which belong to the subgraphs of $X$ induced by $V_1$ and $V_2$, respectively, and put $Z_{1,2}=Z\sm (Z_1\cup Z_2)$.

We first consider the contribution of the vertices in $V_1$ to $f(X)$. 
By Theorem~\ref{thm:ear} and Lemma \ref{lem:ear}, $C_{\leq m-1}$ is connected in $M$ and $r(C_{\leq m-1})= r(X)-|Z|+1$. Hence, by induction,
\begin{equation}\label{eq:degree1}
f(C_{\leq m-1})\leq 2(r(X)-|Z|+2)-n_1.
\end{equation}
For each $v\in V_1$ which is incident to an edge of $Z$ and each base $B'$ of $C_{\leq m-1}$,  we can construct a base $B$ of $X$ with $d_B(v)=d_{B'}(v)+d_{Z}(v)-1$. This implies that  the contribution of the vertices of $V_1$ to $f(X)$ is at most
\begin{equation}\label{eq:degree2}
f(C_{\leq m-1})+2|Z_1|+|Z_{1,2}|-|V(Z)\cap V_1|.
\end{equation}

We next consider the contribution of the vertices of $V_2$ to $f(X)$. Choose $v\in V_2$. Then $v$ is only incident to edges of $Z$ in $X$. Since we can construct a base of $X$ which contains $d_X(v)-1$ edges incident to $v$, the contribution of $v$ to $f(X)$ is $d_X(v)-1$. So the total contribution of the vertices of $V_2$ to $f(X)$ is
\begin{equation}\label{eq:degree3}
 \sum_{v\in V_2}(d_X(v)-1)=2|Z_2|+|Z_{1,2}|-n_2.
\end{equation}
Note that the previous sentence remains valid when $V_2=\emptyset$ since the contribution given by (\ref{eq:degree3}) is zero in this case.

 We may now combine   (\ref{eq:degree1}), (\ref{eq:degree2}) and (\ref{eq:degree3}) with $|Z|=|Z_1|+|Z_2|+|Z_{12}|$ and $n=n_1+n_2$ to obtain
\begin{eqnarray*}
%\label{eq:degree3}
f(X)&\leq& 2(r(X)-|Z|+2)-n_1
+2|Z_1|+|Z_{1,2}|-|V(Z)\cap V_1|+
2|Z_2|+|Z_{1,2}|-n_2\\
&=&2r(X)-|V(Z)\cap V_1|-n+4. 
%\qquad \mbox{(since $|Z|=|Z_1|+|Z_2|+|Z_{12}|$}).
\end{eqnarray*}
Finally we use the hypothesis that the subgraph of $G$ induced by a circuit %in any abstract rigidity matroid on $E(K_n)$ 
is 2-connected. 
%(see, e.g., \cite[Theorem~3.11.11]{GSS93}).
This  and $C_m\cap C_{\leq m-1}\neq \emptyset$ imply that
%Hence
%The fact that $C_1,C_2,\ldots,C_m$ is an ear decomposition gives
$|V(Z)\cap V_1|\geq 2$,
and we obtain $f(X)\leq  2(r(X)+1)-n$ as required.
\end{proof}

Suppose $M$ is a $K_{d+2}$-matroid with $d\geq 2$. Then Lemma \ref{lem:upper_bound} implies that $r_M(X)\leq d|V(X)|-{{d+1}\choose{2}}$ for any connected set $X$ in $M$ and we can now use  Lemma \ref{lem:degree} to deduce that $X$ has a base of minimum degree at most $2d-2$. We next extend this observation to the case when $M$ is an abstract $d$-rigidity matroid and $X$ is cyclic but not necessarily connected.

\begin{lemma}\label{lem:cyclicdegree}
Suppose $M$ is an abstract $d$-rigidity matroid defined on $E(K_n)$ for some $d\geq 2$ and $X$ is a nonempty cyclic set in $M$.
%$K_5$-matroid
Then every base of $X$ has minimum degree at least $d$, and some base of $X$ has minimum degree at most $2d-2$.
\end{lemma}
\begin{proof}
Since $M$ is an  abstract $d$-rigidity matroid, the 0-extension property holds for $M$ and hence each circuit of $M$ induces a subgraph of minimum degree at least $d+1$.  Since $X$ is cyclic, this implies that each vertex of $K_n[X]$ has degree at least $d+1$. 

Suppose $d_B(v)\leq d-1$ for some base $B$ of $X$ and some $v\in V(X)$. Then we can use the 0-extension property to extend $B-v$ to an independent set of size $|B|+1$ in $X$ by adding $d$ edges incident to $v$, contradicting the fact that $B$ is a base of $X$. Hence  $d_B(v)\geq d$ for all bases $B$ of $X$ and all $v\in V(X)$.

%We complete the proof by showing that the fuction $f(X)$ defined in the statement of Lemma \ref{lem:degree} satisfies $f(X)<(2d-1)|V(X)|$.

It remains to show that some base of $X$ has minimum degree at most $2d-2$.
Let $X_1,X_2\ldots,X_q$ be the connected components of $X$ in $M$. Then each $X_i$ is a cyclic set in $M$, $X_i\cap X_j=\emptyset$ for $i\neq j$ and, for any set of bases $B_i$ of $X_i$, $B=\bigcup_{i=1}^q B_i$ is a base of $X$. For $v\in V(X)$, let $a(v)$ be the number of components $X_i$ for which $v\in V(X_i)$. Let $U=\{v\in V(X):a(v)\geq 2\}$ and let $b(X_i)=|U\cap V(X_i)|$ for all $1\leq i\leq q$. Let $n=|V(X)|$ and $n_i=|V(X_i)|$ for all $1\leq i\leq q$. Then 
$\sum_{i=1}^q n_i=n+\sum_{u\in U}(a(u)-1)$ and $\sum_{u\in U}a(u)=\sum_{i=1}^q b(X_i)$. We may apply Lemmas~\ref{lem:upper_bound} and \ref{lem:degree} to each $X_i$ to deduce that
\begin{equation}\label{eq:cyclicdegree}
f(X_i)\leq (2d-1)n_i-(d-1)(d+2).
\end{equation}

Suppose that $b(X_i)\leq d+1$ for some $1\leq i\leq q$. We can use the argument in the first paragraph to deduce that $d_{B_i}(v)\geq d$ for all bases $B_i$ of $X_i$ and all $v\in V(X_i)$.  Together with (\ref{eq:cyclicdegree}) and the hypothesis that $d\geq 2$, this implies that  there are at least $d+2$ vertices $v\in V(X_i)$ such that $v$ has degree at most $2d-2$ in some base of $X_i$. Since $b(X_i)\leq d+1$, we can choose $v\in V(X_i)\setminus U$ such that $v$ has degree at most $2d-2$ in some base $B_i$ of $X_i$. We can now extend $B_i$ to a base $B$ of $X$. Since $v\not\in U$, $v$ will have degree at most $2d-2$ in $B$. 

Hence we may assume that $b(X_i)\geq d+2$ for all $1\leq i\leq q$.
 Since the union of any collection of bases of the $X_i$ gives a base of $X$,  $f(X)=\sum_{i=1}^q f(X_i)$. We may now use Lemma~\ref{lem:degree},
 (\ref{eq:cyclicdegree}) and the fact that $\sum_{i=1}^qn_i=n+\sum_{u\in U}(a(u)-1)$ to deduce that
%\begin{eqnarray*}
%f(X)=\sum_{i=1}^q f(X_i)&<& (2d-1)\sum_{i=1}^q n_i-qd(d+1)\\
%&=&(2d-1)n+(2d-1)\sum_{u\in U}(a(u)-1)-qd(d+1).
%\end{eqnarray*}

\begin{eqnarray*}
(2d-1)n-f(X)&=&(2d-1)n-\sum_{i=1}^q f(X_i)\\
&=&
\sum_{i=1}^q((2d-1)n_i- f(X_i))-(2d-1)\sum_{u\in U}(a(u)-1)\\
&\geq &2\sum_{i=1}^q (dn_i-r(X_i)-1)-(2d-1)\sum_{u\in U}(a(u)-1)\\
&= &2 (dn-r(X)-q)+\sum_{u\in U}(a(u)-1)\\
&= &2 (dn-r(X))-2q-|U|+\sum_{u\in U}a(u).
\end{eqnarray*}
Since $a(u)\geq 2$ for all $u\in U$ and $b(X_i)\geq d+2\geq 5$ for all $1\leq i\leq q$, we have $\sum_{u\in U}a(u)\geq 2|U|$
and 
$\sum_{u\in U}a(u)=\sum_{i=1}^q b(X_i)\geq 5q$. Since $r(X)<dn$ by Lemma \ref{lem:upper_bound}, we have $(2d-1)n-f(X)>0$ and $f(X)<(2d-1)n$. Hence some base of $X$ has minimum degree at most $2d-2$.
\end{proof}

The bound on the minimum degree of a base in Lemma \ref{lem:cyclicdegree} is best possible since the generic $d$-dimensional rigidity matroid contains circuits of minimum degree $2d-1$ when $d\geq 2$. We will see in the next subsection, however, that a stronger bound holds for cyclic flats in ${\cal C}^1_{2,n}$.

\subsection{\boldmath The rank function of ${\cal C}^1_{2,n}$}
Theorem~\ref{thm:K5free} and Lemma~\ref{lem:rank1} tell us that we can bound the rank of any set $X$ in 
%the maximal abstract rigidity matroid
%on $E(K_n)$
$C_{2,n}^1$ by using proper $K_5$-sequences.
%
%
%states that the maximal abstract 3-rigidity matroid is the free elevation of
%the matroid ${\cal R}^{\downarrow}$ of rank ten.
%In ${\cal R}^{\downarrow}$, a set is a non-spanning circuit if and only if it is the edge set of a copy of $K_5$.
%The rank upper bound of Lemma~\ref{lem:rank1} can be applied for the collection of the
%copies of $K_5$.
%Motivated by this, we say that a sequence $(C_1,\dots, C_t)$ of edges sets in $K_n$ is a {\em $K_5$-sequence} if each $C_i$ induces a copy of $K_5$ in $K_n$.
More precisely, Theorem~\ref{thm:K5free} tells us that  $C_{2,n}^1$ is the free-elevation of the rank 10 matroid on $E(K_n)$ whose non-spanning circuits are the edge sets of all copies of $K_5$, and Lemma~\ref{lem:rank1} now implies  that
the rank of a set $X\subseteq E(K_n)$ in $C_{2,n}^1$ is bounded above by
$|X\cup C_{\leq t}|-t$
for any proper $K_{5}$-sequence  $(C_1,\dots, C_t)$.
Our next result shows that this bound is tight for some proper $K_{5}$-sequence in $K_n$. It also establishes an important structural property of cyclic flats in $C_{2,n}^1$, that they always contain a {\em simplicial vertex} i.e.~a vertex whose neighbor set induces a complete subgraph. Its proof uses a simultaneous induction on both statements.
%We first verify this for connected sets in $C_{2,n}^1$.

To simplify notation we put
\[
{\rm val}(F, {\cal C}):=|F\cup C_{\leq t}|-t
\]
for all $F\subseteq E(K_n)$ and  all proper $K_{5}$-sequences  ${\cal C}=(C_1,\dots, C_t)$ contained in $K_n$.

 \begin{theorem}\label{thm:rank}
Suppose $X\subseteq E(K_n)$.\\  (a) The rank of $X$ in 
%The rank function $r$ of 
%the maximal abstract 3-rigidity matroid 
${\cal C}_{2,n}^1$ is given by
\[
r(X)=\min \{|X\cup C_{\leq t}|-t : \text{$(C_1,\dots, C_t)$ is a proper $K_{5}$-sequence in $K_n$} \}.
\]
%for all $X\subseteq E(K_n)$.\\
(b) If $X$ is a cyclic flat in $C_{2,n}^1$, then there exists a vertex $v\in V(X)$ such that  $K(N_X(v))\subseteq X$, and $v$ has degree three in some base of $X$.
\end{theorem}
\begin{proof}
We proceed by contradiction. Suppose that the theorem does not hold for some $X\subseteq E(K_n)$. We may suppose that $X$ has been chosen such that $r(X)$ is as small as possible and, subject to this condition, that $|X|$ is as large as possible. 
%Since (\ref{eq:tight}) holds when $|X|\leq 1$ (with $\cal C$ equal to the empty sequence), we have $|X|\geq 2$.

We first show that (b) holds for $X$ when $X$ is a cyclic flat. Choose a vertex $v\in V(X)$ and a base $B$   of $X$ such that $d_B(v)$ is as small as possible. Then $d_B(v)\in\{3,4\}$  by Lemma \ref{lem:cyclicdegree}.
%, we can choose a vertex $v\in V(X)$ and a base $B$   of $X$ such that . 
Let $N_X(v)=\{v_1,\dots, v_k\}$. Note that $k\geq 4$ since $X$ is cyclic and all circuits of ${\cal C}_{2,n}^1$ have minimum degree at least four.
%If $d_B(v)\leq 2$ then, since $d_X(v)\geq 4$, we could extend $B$ to a larger independent subset of $X$ by adding an edge of $X\setminus B$ incident to $v$. Hence  $d_B(v)\geq 3$.
%and $r(X)=r(X-v)+3$.

Suppose that $d_B(v)= 3$.
%We first show that .
%contains the complete subgraph on $$.
If $K(N_X(v))\not\subseteq X$, then relabelling if necessary, we have $e=v_1v_2\in K(N_X(v))\setminus X$, and  $B-v+e$ is independent since $X$ is a flat.  Then
$B'=(B-v)\cup \{vv_1,vv_2,vv_3,vv_4\}$ is also independent
%in ${\cal C}_{2,n}^1$
% for any
%We perform 1-extension along $e$ (adding new edges $vu_1, vu_2$ and $vu_3, vu_4$ for some
%$u_3, u_4\in N_X(v)\setminus \{u_1,u_2\}$,
 since it can be obtained from $B-v+e$ by a 1-extension.
%Then the resulting set $B-v+ vu_1+vu_2+vu_3+ vu_4$ is an independent set in $X$,  but has a larger size than $B$.
This contradicts the choice of $B$, since $|B'|>|B|$ and $B'\subseteq X$. Hence  $K(N_X(v))\subset X$ and (b) holds for $X$.

It remains to consider the case when $d_B(v)=4$. Then $v$ has degree %four in some base of $X$, and has degree 
at least four in all bases of $X$. We will refer to any base $B$ with the property that $d_B(v)=4$  as a {\em $v$-admissible base} of $X$. We will show that this case cannot occur by first showing that $v$ is a simplicial vertex in $X$ and then deducing that $v$ has degree three in some base of $X$.

\begin{claim}\label{claim:cliquerankX-v} $r(X-v)=r(X)-4$.
%Since $X$ is connected, we have $|N_X(v)|\geq 4$.
\end{claim}
\begin{proof}
The fact that $X$ has a $v$-admissible base implies that $r(X-v)\geq r(X)-4$. Let $B'$ be a base of $X-v$. If $|B'|>r(X)-4$, then $B'$ would extend to a base of $X$ in which $v$ has degree three and would contradict our assumption that $v$ has degree at least four in all bases of $X$. Hence $|B'|= r(X-v)=r(X)-4$.
\end{proof}

As $X$ is a cyclic flat, Corollary~\ref{cor:maximal} tells us that
$X$ is the union of copies of $K_5$.  Hence we may choose $S_1\subseteq X$ such that $S_1$ induces a copy of $K_5$ which contains $v$.
We may assume that
%$S_1$ be any such copy of $K_5$ and put
$V(S_1)=\{v, v_1, v_2, v_3, v_4\}$.

\begin{claim}\label{claim:clique_S1} There exists $Y\subseteq X$ such that $d_Y(v)=5$, $S_1\subseteq Y$ and $Y-vv_i$ is a $v$-admissible base of $X$ for all $1\leq i\leq 4$.
\end{claim}
\begin{proof}
Choose a $v$-admissible base $B$ of $X$. We will use the term {\em base exchange on $B$} to mean the operation which constructs a new base $B+e-f$ by first adding an edge $e\in X\sm B$ to $B$ and then deleting an edge $f$ from the fundamental circuit of $B+e$.
If $vv_i\in B$ for all $1\leq i\leq 4$ then, since $d_B(v)=4$, we could perform a sequence of base exchanges using the edges of $S_1$ to construct a base of $X$ in which $v$ has degree less than four.
 %and contradict Claim\ref{claim:degreefour}.
 Hence we may assume that $vv_4\not\in B$.
%and that $vv_5\in B$ for some $v_5\in N_X(v)\setminus \{v_1,v_2,v_3,v_4\}$.
%We may now assume that
%$vv_i\in B$ for all $1\leq i\leq 3$ then we can again perform could perform a sequence of base exchanges using the edges of $S_1$, to construct a base of $X$ in which $v$ has degree less than four. Hence we may assume that $vv_4\not\in B$  and that $vv_5\in B$ for some $v_5\in N_X(v)\setminus \{v_1,v_2,v_3,v_4\}$.
Since $d_B(v)=4$ and every circuit of ${\cal C}_{2,n}^1$ has minimum degree at least four, we can also use a sequence of base exchanges to ensure that $vv_i\in B$ for all $1\leq i\leq 3$. We can now perform a further sequence of base exchanges using the edges of $S_1$ to ensure that $S_1-vv_4\subseteq B$. Let $Y=B+vv_4$. Then $S_1\subseteq Y$ and $Y-v_1v_4=B$ is a $v$-admissible base of $X$.
The fact that $S_1$ is the fundamental circuit of $B+vv_4$ now implies that $Y-vv_i$ is a $v$-admissible base of $X$ for all $1\leq i\leq 4$.
\end{proof}

%We may suppose that $vv_1, vv_2, vv_3$ exist in $B$ but $vv_4$ dose not.
%(We can first show that at least one edge is missing among $vv_i (i=1,\dots, 4)$, since otherwise by a sequence of base exchanges we can get a base of $X$ in which the degree at $v$ is
%Then using the fact that every circuit has degree at least four, we can convert $B$ such that exactly three edges among the four are contained in $B$.)
%
%By base exchanges, we can further convert $B$ such that $B$ contains a complete subgraph on $\{v_1, v_2, v_3, v_4\}$.
%(Hence $B$ misses only $vv_4$ among the edges in $S_1$.)

Since $d_Y(v)=5$, $Y$ has an edge $vv_5$ for some $v_5\notin V(S_1)$.
In addition,   $Y-v+v_iv_5$ is independent in ${\cal C}_{2,n}^1$ for some $1\leq i\leq 4$,  since otherwise $Y-vv_1$ would be dependent
%for all $1\leq j\leq 4$
(as the closure of $X-v$ would contain  $K(v_1,v_2,\dots, v_5)$ and $v$ is joined to $\{v_1,v_2,\dots, v_5\}$ by four edges in $Y-vv_1$).
Hence, we may  assume without loss of generality that $Y-v+v_4v_5$ is independent.
Since $Y-v$ is a base of $X-v$ by Claims~\ref{claim:cliquerankX-v} and \ref{claim:clique_S1}, this gives
\begin{equation}\label{eq:cliquerankX-v+v4v5}
r(X-v+v_4v_5)=r(X-v)+1=r(X)-3.
\end{equation}

\begin{claim}\label{claim:cliqueNXv}
%\begin{equation}
%\label{eq:claim2}
%\text{
$K(N_X(v))\subseteq {\rm cl}(X-v+v_4v_5)$.
%}
%\end{equation}
\end{claim}
\begin{proof}
We first show  that, for all $u\in N_X(v)-v_5$,
\begin{equation}
\label{eq:cliqueclaim3}
\text{$K(v_1,v_2,v_3,u)\subseteq {\rm cl}(X-v+v_4v_5)$.}
\end{equation}
Suppose to the contrary that $uv_i\not\in {\rm cl}(X-v+v_4v_5)$ for some $u\in N_X(v)-v_5$ and some $1\leq i\leq 3$.
Then $Y-v+v_4v_5+uv_i$ is independent in ${\cal C}_{2,n}^1$. 
Since $K(v_1,v_2,v_3,v_4)\subseteq S_1\subseteq X$, $u\neq v_4$.
Hence $uv_i$ and $v_4v_5$ are disjoint, and we
can perform an $X$-replacement to deduce that $B'=(Y-v)\cup \{vu,vv_i,vv_4,vv_5,vv_j\}$ is independent for any $1\leq j\leq 3$ with $j\neq i$. 
Since $|B'|=|Y|$ and $B'\subseteq X$, this would contradict the fact that $Y-vv_1$ is a base of $X$.

We can now use (\ref{eq:cliqueclaim3}) and axiom (A2) for abstract 3-rigidity from Section \ref{sec:rigidity} to deduce that
${\rm cl}(X-v+v_4v_5)$ contains $K(N_X(v)-v_5)$.  The same axiom implies that we can complete the proof of the claim by showing that $X$ contains at least three edges between $v_5$ and $N_X(v)\setminus \{v_5\}$.
To see this, recall that $X$ is the union of copies of $K_5$.
In particular we have $vv_5\in S_2$ for some copy $S_2$   of $K_5$ in $X$. Since every vertex of $V(S_2)\setminus \{v\}$ is in $N_X(v)$, we have $|(V(S_2)\setminus \{v_5,v\})\cap N_X(v)|=3$,
and $X$ contains the three edges from $v_5$ to the vertices of $V(S_2)\setminus\{v_5,v\}$.
%This complete the proof of (\ref{eq:claim2}).
\end{proof}

%\begin{equation}
%\label{eq:claim3}
%\text{${\rm cl}(X-v+v_4v_5)$ contains an edge between $u$ and $v_i$ for every $u\in N_X(v)\setminus \{v_5\}$ and $1\leq i\leq 3$.}
%\end{equation}
%To see this suppose ${\rm cl}(X-v+v_4v_5)$ misses an edge between $u$ and $v_i$ with $u\in N_X(v)\setminus \{v_5\}$ and $1\leq i\leq 3$.
%Then $B-v+v_4v_5+uv_i$ is independent. Since $uv_i$ and $v_4v_5$ are disjoint,
%and hence we can perform the $X$-replacement along $uv_i$ and $v_4v_5$, which gives a larger independent set of $X$ than $B$. This contradiction confirms (\ref{eq:claim3}).

%By (\ref{eq:claim3}), we can further conclude that ${\rm cl}(X-v+v_4v_5)$ contains the complete subgraph on $N_X(v)\setminus \{v_5\}$.
%Hence to see (\ref{eq:claim2}) it remains to show that $X$ contains at least three edges between $v_5$ and $N_X(v)\setminus \{v_5\}$.
%To see this, recall that $X$ is the union of copies of $K_5$.
%In particular $X$ contains a copy of $K_5$ that contains $vv_5$.
%We denote it by $S_2$.
%Since every vertex of $V(S_2)\setminus \{v\}$ is in $N_X(v)$, we have $|(V(S_2)\setminus \{v_5,v\})\cap N_G(v)|=3$,
%and hence $X$ contains edges between $v_5$ and those three vertices.
%This complete the proof of (\ref{eq:claim2}).

Since $X$ is a counterexample with minimum rank and $r(X-v+v_4v_5)<r(X)$ by (\ref{eq:cliquerankX-v+v4v5}), we can apply  part (a) of the theorem  
%\ref{thm:rank}
%the induction hypothesis
to ${\rm cl}(X-v+v_4v_5)$
to obtain a proper $K_5$-sequence ${\cal C}=(C_1,\dots, C_t)$ with 
\begin{equation}
\label{eq:clique4}
r(X-v+v_4v_5)={\rm val}(\cl(X-v+v_4v_5), {\cal C}).
\end{equation}
Since $K(N_X(v))\subseteq \cl(X-v+v_4v_5)$ by Claim~\ref{claim:cliqueNXv}, no $e\in K(N_X(v))$ can be a coloop in the matroid induced by  $\cl(X-v+v_4v_5)$ and hence 
%
% Choose an edge $e\in K(N_X(v))$.
%By Claim~\ref{claim:cliqueNXv}, $e\in {\rm cl}(X-v+v_4v_5)$. 
%Let $C'$ be a copy of $K_5$ with $e\in C'\subseteq K(N_X(v))$
%and let ${\cal C}'=(C_1,C_2,\ldots,C_t,C')$.
%If $e\not\in C_{\leq t}$, then since $e\in {\rm cl}(X-v+v_4v_5)$, ${\cal C}'$ would be a proper $K_5$-sequence in $K_n$  with
%$${\rm val}(\cl(X-v+v_4v_5), {\cal C}')={\rm val}(\cl(X-v+v_4v_5),{\cal C})-1=r(X-v+v_4v_5)-1$$
%and would contradict Lemma~\ref{lem:rank1}.
%Hence 
$K(N_X(v))\subseteq  C_{\leq t}$ by Lemma~\ref{lem:rank1}.
%In particular,  $v_4v_5\in C_{\leq t}$ and we have ${\rm val}(X-v+v_4v_5,{\cal C})={\rm val}(X-v,{\cal C})$.

Recall that $N_X(v)=\{v_1,\dots, v_k\}$.  Let $C_{t+i}$ be a copy  of $K_5$ on $\{v, v_i, v_{i+1}, v_{i+2}, v_{i+3}\}$ for $i=1, \dots, k-3$, and let ${\cal C}''=(C_1,C_2,\ldots,C_t,C_{t+1},\ldots, C_{t+k-3})$ be the  $K_5$-sequence obtained by
appending $(C_{t+1},\ldots,C_{t+k-3})$ to $\cal C$.
Since $K(N_X(v))\subseteq  C_{\leq t}$, we have $${\rm val}(X,{\cal C}'')={\rm val}(X-v,{\cal C})+k-(k-3)={\rm val}(X-v,{\cal C})+3={\rm val}(X-v+v_4v_5,{\cal C})+3.$$ This gives
\begin{alignat*}{2}
{\rm val}(X,{\cal C}'')&={\rm val}(X-v+v_4v_5,{\cal C})+3  \\
&\leq {\rm val}(\cl(X-v+v_4v_5),{\cal C})+3  \\
&=r(X-v+v_4v_5)+3 & \quad & (\text{by (\ref{eq:clique4})}) \\ 
&=r(X) & & (\text{by (\ref{eq:cliquerankX-v+v4v5})}).
\end{alignat*}
We may now use Lemma \ref{lem:rank1} to deduce that equality must hold throughout and that $C_{\leq t+k-3}\subseteq {\rm cl}(X)$. Since $X$ is a flat and $K(N_X(v))\subseteq  C_{\leq t}$, we have $K(N_X(v))\subseteq %C_{\leq t}\subseteq C_{\leq t+k-3}\subseteq {\rm cl}(X)=
X$ and hence $v$ is a simplicial vertex of $K_n[X]$.

We complete the proof that (b) holds for $X$ by showing that $v$ has degree three in some base of $X$. 
%Since $X$ is cyclic
%%connected in  ${\cal C}_{2,n}^1$ and $|X|\geq 2$, $X$ is cyclic and hence, by 
%Lemma~\ref{lem:clique_nbd} gives $K(N_X(v))\subseteq X$.
Choose a base $B'$ of $X-v$ such that $B'$ contains a base of $K(N_X(v))$. Then $B'$ can be extended to a base $B$ of $X$. The facts that $B'$ contains a base of $K(N_X(v))$ and $|N_X(v)|\geq 4$ (since $X$ is cyclic), imply that $B$ will contain exactly three edges incident to $v$. Hence (b) holds for $X$.

%We next use the fact that (b) holds  when $X$ is a cyclic flat to contradict our initial assumption  (a) of the theorem holds for an arbitrary set $X$.
%As noted before the statement of the theorem, 
%Theorem~\ref{thm:K5free} and Lemma~\ref{lem:rank1} imply that
%\begin{equation}\label{eq:main1}
By Lemma~\ref{lem:rank1},
$r(X)\leq |X\cup C_{\leq t}|-t$ holds
%\end{equation}
for any proper $K_{5}$-sequence  $(C_1,\dots, C_t)$.
Since $X$ is a counterexample to the theorem, (a) does not hold for $X$ and hence 
\begin{equation}\label{eq:tight}
\mbox{$r(X)<{\rm val}(X,{\cal C})$ for all proper $K_{5}$-sequences  ${\cal  C}$ of $K_n$.}
\end{equation}
%Hence it suffices to show that 
%equality is attained in (\ref{eq:main1}) for some proper $K_{5}$-sequence  $(C_1,\dots, C_t)$.
%
%We will show that, for all $X\subseteq E(K_n)$,
%prove by induction on $|V(X)|$ that
%Since (a) does not hold for $X$, we must have
%We willl obtain the required contradiction by showing that
%\begin{equation}\label{eq:tight}
%\mbox{$r(X)={\rm val}(X,{\cal C})$ for some proper $K_{5}$-sequence  ${\cal  C}$ of $K_n$.}
%\end{equation}
Since $r(X)={\rm val}(X,\emptyset)$ when $|X|\leq 1$, (\ref{eq:tight}) implies that $|X|\geq 2$.

\begin{claim}\label{claim:connected}
$X$ is a cyclic flat in ${\cal C}_{2,n}^1$ and $d_X(v)\geq 4$ for all $v\in V(X)$.
\end{claim}
\begin{proof}
Suppose $X$ is not a flat. Then $r(X+e)=r(X)$ for some $e\in E(K_n)\setminus X$. This implies that  $V(X+e)=V(X)$, since the 0-extension property holds for ${\cal C}^1_{2,n}$,  and the maximality of $|X|$ now gives $r(X)=r(X+e)={\rm val}(X+e,{\cal C})\geq {\rm val}(X,{\cal C})$ for some proper $K_{5}$-sequence  $\cal C$ of $K_n$. This contradicts (\ref{eq:tight}).  Hence $X$ is a flat.

Suppose $X$ is not cyclic. Then $r(X-e)=r(X)-1$ for some $e\in X$. By  the minimality of $r(X)$, there is a proper $K_5$-sequence ${\cal C}=(C_{1},\dots, C_{t})$ such that
$r(X-e)={\rm val}(X-e, {\cal C})$. Since $e\not\in \cl(X-e)$, $e\not\in C_{\leq t}$ by 
Lemma~\ref{lem:rank1}, and hence $r(X)=r(X-e)+1={\rm val}(X-e, {\cal C})+1={\rm val}(X, {\cal C})$. This again contradicts  (\ref{eq:tight}).  Hence $X$ is cyclic.
%
%
%
%connected in ${\cal C}_{2,n}^1$. Let   $X_1,\dots, X_k$ be the connected components of ${X}$ in ${\cal C}_{2,n}^1$.
%Then  $X_i\cap X_j=\emptyset$ for all $1\leq i\neq j\leq k$ and $r(X)=\sum_{i=1}^k r(X_i)$. Hence $r(X_i)<r(X)$ for all $1\leq i\leq k$ and, since $X$ is a flat, each $X_i$ is a connected flat.
%By  induction, there is a proper $K_5$-sequence ${\cal C}_i=(C_{1}^i,\dots, C_{t_i}^i)$ such that
%$r(X_i)={\rm val}(X_i, {\cal C}_i)$ for all $1\leq i\leq k$. Lemma~\ref{lem:rank1} and the fact that $X_i$ is a flat now imply that $C^i_{\leq t_i}\subseteq X_i$ for all
%$1\leq i\leq k$.
%
%Since $X_i\cap X_j=\emptyset$ for $i\neq j$, the concatenation ${\cal C}$ of ${\cal C}_1,\dots, {\cal C}_k$
%is a proper $K_5$-sequence with $\sum_{i=1}^k {\rm val}(X_i, {\cal C}_i)={\rm val}(X, {\cal C})$.
%Thus
%$r(X)=\sum_{i=1}^k r(X_i)=\sum_{i=1}^k {\rm val}(X_i, {\cal C}_i)={\rm val}(X, {\cal C})$. This contradicts the fact that  (\ref{eq:tight}) does not hold for $X$.

The assertion that $d_X(v)\geq 4$ for all $v\in V(X)$ now follows since all circuits in ${\cal C}_{2,n}^1$ have minimum degree at least four.
 \end{proof}

Since $X$ is a cyclic flat by Claim~\ref{claim:connected}, (b) holds for $X$ and hence there exists a vertex $v\in V(X)$ such that $K(N_X(v))\subseteq X$, and $v$ has degree three in some base of $X$. Then $r(X-v)=r(X)-3$ and we may apply (a) to $X-v$ to obtain a proper $K_5$-sequence ${\cal C}=(C_1,\dots, C_t)$ with $r(X-v)={\rm val}(X-v, {\cal C})$.

Let $N_X(v)=\{v_1,\dots, v_k\}$ and let $C_{t+i}$ be a copy  of $K_5$ on $\{v, v_i, v_{i+1}, v_{i+2}, v_{i+3}\}$ for $i=1, \dots, k-3$.
Let
${\cal C}''=(C_1,C_2,\ldots,C_t,C_{t+1},\ldots, C_{t+k-3})$ be the  $K_5$-sequence obtained by
appending $C_{t+i}$ to $\cal C$ for $i=1, \dots, k-3$.
Since $K(N_X(v))\subseteq  
%C_{\leq t}
X-v$, we have 
$${\rm val}(X,{\cal C}'')={\rm val}(X-v,{\cal C})+k-(k-3)={\rm val}(X-v,{\cal C})+3
%$.
%This gives
%${\rm val}(X,{\cal C}'')={\rm val}(X-v,{\cal C})+3
=r(X-v)+3=r(X).$$
%We may now use Lemma \ref{lem:rank1} to deduce that equality must hold and hence that 
This contradicts (\ref{eq:tight}) and completes the proof of the theorem.
\end{proof}

\begin{corollary}\label{cor:good_char}
The problem of deciding whether a given set $X\subseteq E(K_n)$  is independent  in  
%the maximal abstract 3-rigidity matroid 
${\cal C}_{2,n}^1$ is in ${\sf NP}\cap {\sf coNP}$.
\end{corollary}
\begin{proof} Theorem \ref{thm:rank} immediately implies that the problem belongs to  ${\sf coNP}$. The fact that it also belongs to ${\sf NP}$ follows by applying either the Schwartz-Zippel Lemma \cite{S80, Z79}, or the more elementary coordinate fixing argument of Fekete and Jord\'an \cite{FJ} to the $C_2^1$-cofactor matrix.
\end{proof}

\section{Shellable covers}\label{sec:shell}

We use Theorem \ref{thm:rank} to obtain an alternative formula for the rank function of 
%the maximal abstract 3-rigidity matroid 
${\cal C}_{2,n}^1$ which is closely related to existing conjectures on the rank functions of the generic 3-dimensional rigidity matroid  \cite{DDH83,JJ06,JDIMACS} and the generic 2-dimensional matrix completion matroid \cite{JJT14}.

Let
%$F\subseteq E(K_n)$.
%A
%%{\em $2$-thin cover}
%{\em cover}
%of $F$
%is a family
$\MX$ be a family of subsets of $V(K_n)$.
%, each of cardinality at least five. 
%The family $\MX$ is said to {\em $t$-thin} if $|X_i\cap X_j|\leq t$ for all distinct $X_i,X_j\in \MX$.
%
%such that each edge of $F$ is induced
%by at least one set in $\MX$.
% and $|X_i\cap X_j|\leq 2$ for all distinct $X_i,X_j\in \MX$.
A {\em hinge} of $\MX$ is
 a pair of vertices $\{x,y\}$
with  $X_i\cap X_j=\{x,y\}$ for two distinct  $X_i,X_j\in \MX$. Let $H(\MX)$ be the set of all hinges of $\MX$. 
%The {\em hinge graph} of $\MX$ is the bipartite graph with bipartition $(\MX,H(\MX))$ in which $X_i$ and $h$ are incident if $h$ is contained in $X_i$. 
The {\em degree} $\deg_\MX(h)$ of a hinge $h$ of $\MX$
 is the number of sets in $\MX$ which contain $h$. 
 %given by its degree in the hinge graph of $\MX$. 
 The family $\MX$ is said to be {\em $t$-thin} if $|X_i\cap X_j|\leq t$ for all distinct $X_i,X_j\in \MX$.

For $F\subseteq E(K_n)$, we say that $\MX$ is a {\em cover} of $F$
if each set in $\MX$ has cardinality at least five and each edge of $F$ is induced by at least one set in $\MX$. 
%(Note that we will restrict our attention to covers $\MX$ for which every set in $\MX$ has size at least five.). 
Dress et al~\cite{DDH83} conjectured that 2-thin covers could be used to characterize the rank function of  
%the 3-dimensional rigidity matroid 
${\cal R}_{3,n}$. More specifically they defined
the {\em value} of a family $\MX$ as
$$\valD(\MX)= \sum _{X\in \MX}(3|X|-6)\,-\sum_{h\in H(\MX)}(\deg_\MX(h)-1)$$
and conjectured that the rank of any $F\subseteq E(K_n)$ in ${\cal R}_{3,n}$ is given by $\min\{|F_0|+{\valD}(\MX)\}$ where the minimum is taken over all $F_0\subseteq F$ and all 2-thin covers $\MX$ of $F\setminus F_0$ with sets of size at least three. This conjecture was shown to be false in \cite{JJ05} by giving an example of a 2-thin family $\MX$ with a negative value. Modified conjectures were given in \cite{JJ06, JDIMACS} which placed further restrictions on the type of cover used to obtain the minimum.
%In particular the concept of a `$k$-degenerate cover' was introduced in \cite{JJT14, JDIMACS}. 
We will show that these restricted 2-thin covers can be used to characterize the rank function of the maximal abstract 3-rigidity matroid ${\cal C}^1_{2,n}$. 
%(for a slightly modified definition of a degenerate cover).

A family $\MX$ of subsets of $V(K_n)$ 
%of size at least five 
is said to be {\em $k$-shellable} if its elements  can be ordered as a sequence $(X_1,X_2,\ldots,X_m)$ so that, for all $2\leq i\leq m$, 
%the set of hinges of the subfamily $\{X_1,X_2,\ldots,X_{i}\}$ which are contained in $X_{i}$ covers at most $k$ vertices of $X_{i}$.
$|X_i\cap \bigcup_{j=1}^{i-1}X_j|\leq k$.
Similarly,  $\MX$  is said to be {\em $k$-degenerate} if its elements  can be ordered as a sequence $(X_1,X_2,\ldots,X_m)$ so that, for all $2\leq i\leq m$, the set of hinges of the subfamily $\{X_1,X_2,\ldots,X_{i}\}$ which are contained in $X_{i}$ has size at most $k$. The two concepts are 
%closely 
related since 
%every $k$-degenerate sequence is $2k$-shellable, and 
every $k$-shellable family is ${{k}\choose{2}}$-degenerate. 

\begin{theorem}\label{thm:rankcover}
Suppose $r$ is the rank function of ${\cal C}^1_{2,n}$ and $F\subseteq E(K_n)$.
Then
\begin{equation}\label{eq:rankcover}
%$$
r(F)=\min\left\{|F_0|+\valD(\MX) \right\},
%$$
\end{equation}
where the minimum is taken over all $F_0\subseteq F$ and all $4$-shellable,
$2$-thin
covers $\MX$ of $F\setminus F_0$  with sets of size at least five.
% with sets of size at least five.
\end{theorem}

%Equivalently, $\MX$ is $k$-degenerate if the  hinge graph  of $\MX$ can be reduced to the empty graph by recursively choosing a set $X\in \MX$ of degree at most $k$ in the hinge graph of $\MX$
% %in the hinge graph of $\MX$
% and then taking the hinge graph of $\MX-X$ i.e. we delete the vertex $X$ from the hinge graph of $\MX$ and also delete any hinge vertices which have degree one in the resulting subgraph.

It was conjectured in \cite{JDIMACS} that the 
expression on the right hand side of
%max/min formula 
(\ref{eq:rankcover}) determines the rank function of 
%the generic 3-dimensional rigidity matroid 
${\cal R}_{3,n}$ when the minimum is taken over all $9$-degenerate 2-thin covers of $F\setminus F_0$  with sets of size at least three.
Theorem~\ref{thm:rankcover} solves the $C_2^1$-rigidity counterpart of this conjecture (since 4-shellable covers are 6-degenerate). 

%Theorem \ref{thm:rankcover} and Lemma \ref{lem:valupper} below imply that this conjectured rank  formula is valid for the maximal abstract 3-rigidity matroid ${\cal C}_{2,n}^1$, and indeed that we may take the minimum over all $9$-degenerate 2-thin covers.

The proof of Theorem \ref{thm:rankcover} will take up the remainder of this section.
Our first result (Lemma~\ref{lem:valupper}) implies that both 4-shellable covers and 9-degenerate covers give an upper bound on the rank function of any 
%abstract $3$-rigidity 
$K_5$-matroid. 
%The second result (Theorem~\ref{thm:dress}) shows that the equality indeed holds for some 4-shellable 2-thin cover. We do this by solving the $C_2^1$-counterpart of another conjecture on the generic 3-dimensional rigidity matroid.  

Given a
%an abstract 3-rigidity 
matroid $M$ on $E(K_n)$ we say that a family $\MX$ of subsets of $V(K_n)$ 
%of size at least five 
is {\em $M$-degenerate} if  its elements  can be ordered as a sequence $(X_1,X_2,\ldots,X_m)$ so that, for all $2\leq i\leq m$, the set of hinges of the subfamily $\{X_1,X_2,\ldots,X_{i}\}$ which are contained in $X_{i}$ is independent in $M$ (when viewed as a set of edges of $K_n$).

\begin{lemma}\label{lem:valupper}
Let $M$ be a $K_5$-matroid
%an abstract 3-rigidity matroid 
defined on $E(K_n)$ and $r$ be its rank function.
Suppose $F\subseteq E(K_n)$ and $\MX$ is an  $M$-degenerate
%$2$-thin
cover of $F$  with sets of size at least five. Then $r(F)\leq \valD(\MX)$.
\end{lemma}
\begin{proof} We use induction on $|\MX|$.
If $\MX=\{X\}$ then $H(\MX)=\emptyset$ and Lemma \ref{lem:upper_bound} gives
$$r(F)\leq \max\{3|V(F)|-6,1\}\leq 3|X|-6=\valD(\MX)$$
since $V(F)\subseteq X$ and $|X|\geq 5$.
Hence suppose that $|\MX|\geq 2$.

Let  $H$ be the set of all edges of $K_n$ which are induced by a hinge of $\cal X$.
We may assume that $H\subseteq  F$  since adding edges of $H$ to $F$ will not change $\valD(\MX)$ or the fact that $\MX$ is an $M$-degenerate cover of $F$, and can only increase $r(F)$.
Since  $\MX$ is $M$-degenerate, we can choose an ordering
$(X_1,X_2,\ldots,X_m)$ of $\MX$ so that, for all $2\leq i\leq m$, the set of hinges of the subfamily $\{X_1,X_2,\ldots,X_{i}\}$ which are contained in $X_{i}$ are independent in $M$ (when viewed as a set of edges of $K_n$).
Let $F_m$ be the set of edges in $F$ covered by $X_m$, $H_m=F_m\cap H$ and
%cover of $F+uv$, and  $r(F)\leq r(F+uv)$.
$F'=F\setminus (F_m\sm H_m)$. Since $H_m$ is independent in $M$, we may choose a maximum independent subset  $B'$ of $F'$  with $H_m\subseteq B'$.
Since $\MX'=\MX\setminus\{X_m\}$ is an $M$-degenerate
%$2$-thin
cover of $F'$, we may use induction to deduce that $|B'|\leq \valD(\MX')$.  Let $B$ be a maximum independent subset of $F$ which contains $B'$. Then $|B\cap F_m|\leq 3|X_m|-6$ and hence $|B\setminus B'|\leq 3|X_m|-6-|H_m|$.  This gives
$$r(F)=|B|= |B'|+|B\setminus B'|\leq \valD(\MX')+ 3|X_m|-6-|H_m|=\valD(\MX).$$
\end{proof}

Lemma \ref{lem:valupper} and
 the facts that 4-shellable covers are $M$-degenerate in any 
 %abstract rigidity 
 $K_5$-matroid on $E(K_n)$ and  $r(F)\leq |F_0|+
r(F\setminus F_0)$  for all $F_0\subseteq F\subseteq E(K_n)$, imply that the right hand side of (\ref{eq:rankcover}) gives an upper bound on $r(F)$.  It follows that we may complete the proof of Theorem \ref{thm:rankcover} by exhibiting a set $F_0\subseteq F$ and a 4-shellable, 2-thin cover $\MX$ of $F\setminus F_0$ with sets of size at least five, such that $r(F)=|F_0|+\valD(\MX)$. To do this, 
we will choose sets $F_0$ and $\MX$ suggested by
a second conjecture  of 
%on the generic 3-dimensional rigidity matroid ${\cal R}_{3,n}$ which was made by  
Dress, 
%at a conference on rigidity held in Montreal in 1987,
see \cite{CDT,GSS93,T99}.

Given a graph $G=(V,E)$, a {\em maximal clique} of $G$ is a maximal subset $X\subseteq V$ such that $G[X]$ is a complete graph.
Suppose that $F\subseteq E(K_n)$ 
%and let $\bar F$ be the closure of $F$ in ${\cal R}_{3,n}$.
is a flat in ${\cal R}_{3,n}$.
Let $\MX^*$ be the set of all  maximal cliques in $K_n[F]$ of size at least three, and $F_0$ be the set of edges in $F$ which are not covered by $\MX^*$.
Dress's second conjecture is that the rank of $F$ in ${\cal R}_{3,n}$ is equal to $|F_0|+\valD(\MX^*)$. We will show that a slightly modified form of this conjecture holds for 
%the maximal abstract 3-rigidity matroid 
${\cal C}_{2,n}^1$.

\begin{thm}\label{thm:dress}
Suppose that $F\subseteq E(K_n)$ is a flat in $\MC_{2,n}^1$. Let $\MX^*$ be the set of all maximal cliques in $K_n[F]$ of size at least five, and $F_0$ be the set of edges in $F$ which are not covered by $\MX^*$.
Then $\MX^*$ is a $4$-shellable, $2$-thin cover of $F\setminus F_0$ and
\begin{equation}\label{eq:dressrank}
r(F)=|F_0|+\valD(\MX^*).
\end{equation}
\end{thm}
\begin{proof} We proceed by contradiction. Suppose that the theorem does not hold for some $F\subseteq E(K_n)$ and that $F$ has been chosen to be as small as possible. The hypothesis that $F$ is a flat and the axiom (A2) of abstract 3-rigidity imply that $\MX^*$ is 2-thin. We will obtain a contradiction by showing  that  $\MX^*$ is  $4$-shellable and
$r(F)=|F_0|+\valD(\MX^*)$.

We first show that $F$ is a cyclic set in ${\cal C}_{2,n}^1$.  Suppose to the contrary that some $e\in F$ is a coloop in  ${\cal C}_{2,n}^1|_F$. Then $e\in F_0$. Since $F$ is a flat, $F-e$ is a flat and $\MX^*$ is the set of all maximal cliques in $K_n[F-e]$ of size at least five. By 
%induction 
the minimality of $F$, $\MX^*$ is 4-shellable and we have
$$r(F)=r(F-e)+1=|F_0-e|+\valD(\MX^*)+1=|F_0|+\valD(\MX^*).$$
This contradicts the choice of $F$. Hence $F$ is cyclic. Corollary \ref{cor:maximal} now implies that $F_0=\emptyset$.

%We next show that $F$ is a connected set in ${\cal C}_{2,n}^1$.  Suppose not. Let $F_1,F_2,\ldots F_t$ be the connected components of $F$ and let $\MX_i$ be  the set of all maximal cliques in $K_n[F_i]$ of size at least five. Since no circuit of ${\cal C}_{2,n}^1$ in $F$ can contain an edge of  both $F_i$ and $F_j$ for $i\neq j$, $\{\MX_1,\ldots,\MX_t\}$ is a partition of $\MX^*$. In addition, for each $X\in \MX_i$ and $X'\in
%\MX_j$ with $i\neq j$, we have $|X\cap X'|\leq 1$, and hence $\{H(\MX_1),\ldots,H(\MX_t)\}$ is a  partition of $H(\MX^*)$.  We can now use induction to deduce that $\MX_i$ is 4-shellable and $r(F_i)=\valD(\MX_i)$ for all $1\leq i\leq t$. This implies that $\MX^*$ is 4-shellable and
%$$r(F)=\sum_{i=1}^t r(F_i)=\sum_{i=1}^t \valD(\MX_i)=\valD(\MX^*).$$
%This contradicts the choice of $F$ and hence $F$ is a connected flat.

By Theorem~\ref{thm:rank}(b),
% Corollary \ref{cor:clique_3nbd}, 
there exists a base $B$ of $F$ and a vertex $v\in V(F)$ such that $d_B(v)=3$ and $K(N_F(v))\subseteq F$. Then $r(F-v)=r(F)-3$ and $\cl(F-v)=\cl(F)-v=F-v$ so $F-v$ is a flat in $\MC_{2,n}^1$. Since  $K(N_F(v))\subseteq F$ and $d_F(v)\geq 4$, $v$ is contained in a unique maximal clique $X_v\in \MX^*$.

Suppose that $d_F(v)\geq 5$ and let $X_v'=X_v-v$. Then $|X_v|\geq 6$, so $|X'_v|\geq 5$, $\MX'=\MX^*-X_v+X_v'$ is the set of maximal cliques of $K_n[F-v]$, and $H(\MX')=H(\MX^*)$. We can now use the minimality of $F$ to deduce that $\MX'$ and $\MX^*$ are 4-shellable and
$$r(F)=r(F-v)+3=\valD(\MX')+3=\valD(\MX^*).$$
This contradicts the choice of $F$.

It remains to consider the case when $d_F(v)= 4$. Then  $|X_v|=5$ and
$\MX'=\MX^*-X_v$ is the set of maximal cliques of $K_n[F-v]$ of size at least five. Furthermore, the set $F_0'$ of edges of $F-v$ which do not belong to a clique in $\MX'$ is given by $K(N_X(v))\setminus H(\MX^*)$.  We can now use induction to deduce that $\MX'$ is 4-shellable and $r(F-v)=|F_0'|+\valD(\MX')$. Let $(X_1,X_2,\ldots,X_t)$ be a 4-shellable ordering of the cliques in $\MX'$. Since $|X_v\cap V(F-v)|=4$, $(X_1,X_2,\ldots,X_t,X_v)$ will be a 4-shellable ordering of the cliques in $\MX^*$ and hence $\MX^*$ is 4-shellable. In addition
$$r(F)=r(F-v)+3=|F_0'|+\valD(\MX')+3=\valD(\MX^*).$$
This contradicts the choice of $F$ and completes the proof of the theorem.
\end{proof}

As noted above, Theorem \ref{thm:rankcover} follows from Lemma \ref{lem:valupper} and Theorem \ref{thm:dress}: Lemma \ref{lem:valupper} implies that the minimum on the right hand side of (\ref{eq:rankcover}) is an upper bound on $r(F)$, and we may apply 
Theorem \ref{thm:dress} to $\cl(F)$ to deduce that equality is attained.

\medskip
%A rather more complicated definition of an 
Jackson and Jord\'an introduced $M$-degenerate 2-thin covers as `{iterated $2$-thin covers}' in \cite{JJ06} and conjectured   that they determine the rank function of 
 %the 3-dimensional rigidity matroid
 ${\cal R}_{3,n}$. Theorem \ref{thm:rankcover}, Lemma \ref{lem:valupper} and the fact that $4$-shellable covers are $M$-degenerate  for any 
 %abstract rigidity 
 $K_5$-matroid imply that the conjectured rank formulae  in \cite[Conjectures 3.2, 3.3]{JJ06} hold for the $C_2^1$-cofactor matroid. 
 
Cheng and Sitharam introduced  $k$-degenerate 2-thin covers in \cite{CS14} as `generalized partial $k$-trees'. They used 3-degenerate 2-thin covers to show that the number of edges in any maximal `(3,6)-sparse subgraph' of a graph gives an upper bound on its rank in 
%the 3-dimensional rigidity matroid
${\cal R}_{3,n}$.

%\section{Applications}

\section{\boldmath Sufficient connectivity conditions for $C_2^1$-rigidity of graphs}\label{sec:con}
We say that a graph $G=(V,E)$ with $n$ vertices is {\em $C_2^1$-rigid} if its edge set spans ${\cal C}_{2,n}^1$. It is {\em $k$-connected} if $n\geq k+1$ and $G-U$ is connected for all $U\subset V$ with $|U|\leq k-1$. 
%It is {\em essentially $k$-connected} if it is $k-1$ connected and, if $G-U$ is disconnected for some $U\subset V$ with $|U|= k-1$, then $G-U$ has exactly two components, one of which is an isolated vertex.
We use the results of the preceding section to obtain sufficient connectivity conditions for a graph to be $C_2^1$-rigid.

Jackson and Jord\'an 
%\cite{JJ06} showed that, if $G=(V,F)$ is 3-connected and $F$ is a cyclic set in the generic 2-dimensional rigidity matroid then $G$ is rigid in $\R^2$.
%They 
conjectured in \cite[Example 2]{JJ06} that, if $G=(V,F)$ is 5-connected and $F$ is a circuit in ${\cal R}_{3,n}$, then $G$ is rigid in $\R^3$. Our first result implies that the $C_2^1$-rigidity counterpart of this conjecture is true.  
%
%Corollary~\ref{cor:maximal} (which tells us that every edge in a cyclic flat belongs to a copy of $K_5$) and Theorem~\ref{thm:dress}. We will show the same conclusion holds under the weaker hypothesis of essential 5-connectivity in Theorem  \ref{thm:12conn} below.

\begin{theorem}
\label{thm:5conn_conn}
Suppose $G=(V,F)$ is a 5-connected graph on $n$ vertices and 
$F$ is a cyclic set in ${\cal C}_{2,n}^1$.   Then $G$ is $C_{2}^1$-rigid. %and $F$ is a connected set in ${\cal C}_{2,n}^1$.
\end{theorem}
\begin{proof} We may assume without loss of generality that $F$ is a flat since replacing $F$ by its closure will not affect the hypotheses or conclusion of the theorem. Corollary \ref{cor:maximal} now implies that every edge of $G$ belongs to a copy of $K_5$.
Let ${\cal X}$ be the set of all maximal cliques of size at least five in $H$. By Theorem~\ref{thm:dress},  ${\cal X}$ is a 4-shellable  cover of $F$. The hypothesis that $G$ is $5$-connected now implies that
$|{\cal X}|=1$. Hence $G$ is a clique with at least six vertices and the theorem holds for $G$. 
\end{proof} 

Tay \cite{T93} gave examples of 4-connected graphs $G=(V,F)$ on $n$ vertices which are not rigid in $\R^3$ and are such that $F$ is a circuit 
in ${\cal R}_{3,n}$. The same  graphs also show that the connectivity condition of Theorem \ref{thm:5conn_conn} is best possible since they are not ${\cal C}_{2,n}^1$-rigid and $F$ is a circuit  in ${\cal C}_{2,n}^1$.

%
%
%
%\begin{theorem}
%\label{thm:5conn_conn}
%Suppose $F\subseteq E(K_n)$ is a cyclic set in ${\cal C}_{2,n}^1$ and the graph $G=(V(K_n),F)$ is essentially $5$-connected. Then $F$ is a connected set in ${\cal C}_{2,n}^1$ and $G$ is $C_{2}^1$-rigid.
%\end{theorem}
%\begin{proof} We may assume without loss of generality that $F$ is a flat in ${\cal C}_{2,n}^1$ since replacing $F$ by its closure will not affect the hypotheses or conclusions of the theorem.  Let ${\cal X}$ be the set of all maximal cliques in $G$. Then Theorem ~\ref{thm:dress} and the fact that  every edge of $F$ belongs to a copy of $K_5$ in $G$ by Corollary \ref{cor:maximal},  imply that ${\cal X}$ is 4-shellable and $r(F)=\val({\cal X})$. Let $(X_1,X_2,\ldots,X_m)$ be a `4-shellable' ordering of $\cal X$. Then $|X_m\cap \bigcup_{i=1}^{m-1}X_i|\leq 4$ and the hypothesis that $G$ is essentially $5$-connected now tells us that $G[X_m]=K_5$.  
%
%Let $v$ be the unique vertex in  $X_m\sm \bigcup_{i=1}^{m-1}X_i$.
%Then $d_F(v)=4$. We will show that $F-v$ is cyclic. NOT CLEAR HOW TO DO THIS. MAYBE EVERY CONNECTED COMPONENT OF $F$ INDUCES A COPY OF $K_5$.
%\end{proof} 

Lov{\'a}sz and Yemini~\cite{LY82} proved that every 6-connected graph is rigid in $\R^2$ and conjectured that every 12-connected graph is rigid in $\R^3$.
Our second result solves the $C_2^1$-rigidity counterpart of their conjecture.

\begin{theorem}\label{thm:12conn}
Let $G=(V,F)$ be a $12$-connected graph on $n$ vertices and $S\subset E$ with $|S|\leq 6$.
Then $G-S$ is $C_{2}^1$-rigid.
\end{theorem}
\begin{proof}
%We first prove (a).  
Suppose, for a contradiction, that $G-S$ is not $C_{2}^1$-rigid. We may assume that $G$ and $S$ have been chosen to be a counterexample such that $|V|$ is as small as possible and, subject to this condition, $|F\setminus S|$ is as large as possible.
%Suppose that the statement is false. We take a counterexample $G$ such that $|V(G)|$ is as small as possible.
Then $F\setminus S$ is a flat in ${\cal C}_{2,n}^1$ since if $e\in \cl(F\setminus S)\setminus (F\setminus S)$, then we may apply induction to $(F+e,S-e)$ to deduce that $G-S+e$ is $C_{2}^1$-rigid and hence, since $e\in \cl(F\setminus S)$, $G-S$ is  $C_{2}^1$-rigid.

Let ${\cal X}$ be the set of all maximal cliques of size at least five in $G-S$, $\hat F_0$ be the set of edges in $G-S$ not covered by any member of ${\cal X}$ and put $F_0=\hat F_0\cup S$.
By Theorem~\ref{thm:dress}, ${\cal X}$ is a 4-shellable 2-thin cover of $F\setminus  F_0$ and 
$|\hat F_0|+\valD({\cal X})=r(F\setminus S)<3|V(G)|-6$. Since $|S|\leq 6$ this gives
\begin{equation}
\label{eq:connectivity_rank}
|F_0|+\valD({\cal X})<3|V(G)|.
\end{equation}

For each vertex $v\in V(F)$, let $F_0^v$ be the set of edges in $F_0$ incident to $v$, $k_v$ be the number and ${\cal X}^v=\{X_1^v, \dots, X_{k_v}^v\}$ be the set of cliques in $\cal X$ which contain $v$, 
%${\cal X}^v=\{X_1^v, \dots, X_{k_v}^v\}$ be the subfamily of ${\cal X}$ consisting of the cliques that contain $v$, 
and 
%Let $k_v=|{\cal X}^v|$.
$H_v$ be the set of hinges of ${\cal X}^v$ that contain $v$.  
Note that $H_v\subseteq H({\cal X})$ and ${\rm deg}_{{\cal X}^v}(h)={\rm deg}_{\cal X}(h)$ for each $h\in H_v$.

\begin{claim}
\label{claim:connectivity1}
For each $v\in V(F)$, either $F_0^v\neq \emptyset$ or $k_v\geq 2$.
\end{claim}
\begin{proof}
Suppose that $F_0^v=\emptyset$ and $k_v=1$, i.e., $v$ is not incident to an edge of $F_0$ and is contained in exactly one maximal clique  $X_1^v\in {\cal X}$. 
Since $\cal X$ covers $F\sm F_0$, we have $K(N_G(v))\subset G$ and  either $G-v$ is  12-connected or $G=K_{13}$.
Since $K_{13}$ is $C_2^1$-rigid, the first alternative must occur, and we may apply induction to deduce that $G-v-S$ is $C_2^1$-rigid.
This and $d_G(v)\geq 12$ in turn imply that $G$ is $C_2^1$-rigid, which is a contradiction.
\end{proof}

The $C_2^1$-rigidity of $G-S$ will  follow easily from our next claim.

\begin{claim}
\label{claim:connectivity2}
For all $v\in V(G)$,
\begin{equation}
\label{eq:connectivity_key}
\frac{|F_0^v|}{2}+\sum_{i=1}^{k_v}\left(3-\frac{6}{|X_i^v|}\right)-\sum_{h\in H_v}\left(\frac{{\rm deg}_{\cal X}(h)-1}{2}\right)\geq 3.
\end{equation}
\end{claim}
\begin{proof}
If $k_v=0$, then, by 12-connectivity, $|F_0^v|\geq 12$ holds, which implies (\ref{eq:connectivity_key}).
Hence we may suppose that $k_v\geq 1$.
%By Claim~\ref{claim:connectivity1}, $|F_0|\geq 1$ or $k_v\geq 2$.
%The proof of (\ref{eq:connectivity_key}) is done by a case analysis.
%
%\medskip
%
%Case 2: $k_v=0$. By 12-connectivity, $|F_0^v|\geq 12$, which implies (\ref{eq:connectivity_key}).
%
%\medskip

%\medskip

Suppose that $k_v=1$, i.e., ${\cal X}^v=\{X_1^v\}$. Then $H_v=\emptyset$.
By Claim~\ref{claim:connectivity1}, $|F_0^v|\geq 1$.
Hence, if $|X_1^v|\geq 12$, then (\ref{eq:connectivity_key}) holds.
It remains to  consider the subcase when $5\leq |X_1^v|\leq 11$.
As $v$ has degree at least 12 in $G$, we have 
$(|X_1^v|-1)+|F_0^v|\geq 12$.
%By this equation, the left side of (\ref{eq:connectivity_key}) satisfies
This implies that
\[
\frac{|F_0^v|}{2}+\left(3-\frac{6}{|X_1^v|}\right)
\geq 9.5-\frac{|X_1^v|}{2}-\frac{6}{|X_1^v|}.
\]
When $5\leq |X_1^v|\leq 11$, the right side of this inequality is at least $3$,
and we get (\ref{eq:connectivity_key}).

%\medskip

Hence we may assume that $k_v\geq 2$.
Put $c_1=0$ and let $c_i=\left|\left(\bigcup_{j=1}^{i-1}X_j^v\right)\cap X_i^v\right|-1$ for $2\leq i \leq k_v$. 
Then $c_i$ represents the contribution of $X_i^v$ to the `hinge count' at $v$ in
%for the sequence 
$\valD(X_1^v, \dots, X_{i}^v)$
%constructing  $\cup {\cal X}^v$ sequentially from $i=1$).
%Then observe that 
and we have
\[
\sum_{h\in H_v}({\rm deg}_{{\cal X}^v}(h)-1)=\sum_{i=1}^{k_v}c_i.
\]
Since ${\cal X}$ is 4-shellable, ${\cal X}^v$ is 4-shellable.
Hence, by reordering the sequence $(X_1^v,\dots, X_{k_v}^v)$ if necessary, we can suppose that $c_i\leq 3$ for all $2\leq i\leq k_v$. 
Also, since ${\cal X}^v$ is 2-thin, 
%$c_1=0$ and 
$c_2\leq 1$.
Since $|X_i^v|\geq 5$ for all $1\leq i\leq k_v$, we may deduce that
\[
\sum_{i=1}^2 \left(3-\frac{6}{|X_i^v|}-\frac{c_i}{2}\right)\geq 
%\left(3-\frac{6}{|X_1^v|}\right)+\left(3-\frac{6}{|X_2^v|}-\frac{c_2}{2}\right)\geq
(3-\tfrac{6}{5})+(3-\tfrac{6}{5}-\tfrac{1}{2})>3,
\]
and, for all $3\leq i\leq k_v$,
\[
3-\frac{6}{|X_i^v|}-\frac{c_i}{2}\geq 3-\tfrac{6}{5}-\tfrac{3}{2}> 0.
\]
Combining these inequalities, we obtain
%\begin{align*}
$$
\frac{|F_0^v|}{2}+\sum_{i=1}^{k_v}\left(3-\frac{6}{|X_i^v|}\right)-\sum_{h\in H_v}\left(\frac{{\rm deg}_{\cal X}(h)-1}{2}\right)
%&
= \frac{|F_0^v|}{2}+\sum_{i=1}^{k_v}\left(3-\frac{6}{|X_i^v|}-\frac{c_i}{2}\right)>3. 
%\\
%&\geq \sum_{i=1}^{2}\left(3-\frac{6}{|X_i^v|}-\frac{c_i}{2}\right) \\
%&\geq 3.1>3.
%\end{align*}
$$
This completes the proof of the claim.
\end{proof}

Taking the sum of (\ref{eq:connectivity_key}) over all vertices of $G$, we obtain
\[
|F_0|+\valD({\cal X})=|F_0|+\sum_{X\in {\cal X}}(3|X|-6)-\sum_{h\in H({\cal X})} ({\rm deg}_{\cal X}(h)-1)\geq 3|V(G)|.
\]
This contradicts (\ref{eq:connectivity_rank}) and completes the proof of the theorem. 
%Hence $G-S$ is $C_2^1$-rigid.
% and (a) holds.
%
%We next prove (b).
% Since $(G-T)-e$ is $C_2^1$-rigid for all $e\in F\setminus T$ by (a), $F\setminus T$ is a cyclic set in ${\cal C}_{2,n}^1$. Since $|T|\leq 5$, $G-T$ is $7$-connected we may now apply Theorem \ref{thm:5conn_conn} to deduce that $F\setminus T$ is a connected set.
\end{proof}

Lov\'asz and Yemini  \cite{LY82} gave examples of 11-connected graphs $G=(V,F)$  which are not rigid in $\R^3$.  These graphs also show that the connectivity condition of Theorem \ref{thm:12conn} is best possible since they are not ${\cal C}_{2,n}^1$-rigid. The 12-connected graph $G$ consisting of two large complete graphs joined by a set $T$ of 12 disjoint edges fails to be  ${\cal C}_{2,n}^1$-rigid if we delete 7 edges from $T$. This shows that the bound on $|S|$ in Theorem \ref{thm:12conn} is best possible.

\section{Problems and Remarks}\label{sec:close}

%A lot of 
Many interesting problems remain.
\begin{enumerate}
\item Developing a deterministic polynomial time algorithm to solve the combinatorial optimization problem of evaluating  the rank formulae in Theorems~\ref{thm:rank} and \ref{thm:rankcover} is our most obvious
%important %question.
open problem.
These formulae imply that the decision problem of determining whether the rank takes any given value is in NP$\cap$co-NP, so it is likely that 
%an efficient 
such an algorithm exists.
But it is not clear
%, however, whether this problem can be captured by 
how to use the existing theory of submodular functions to design such an algorithm since our rank formulae 
%seem to have a form which is 
are significantly different to other known matroid rank formulae.

%\item Conjecture~\ref{conj:max_cofactor} may be true for larger $d$, but the proof strategy in \cite{CJT1} for $d=3$ does not extend to $d\geq 4$.

\item  Let $M_0$ be the rank ${{d+2}\choose{2}}$ truncation of ${\cal R}_{d,n}$ (or equivalently, of any other $K_{d+2}$-matroid on $E(K_n))$ for $n\geq d+2$. 

\begin{conjecture}\label{con:K_d+2}
%We conjecture that 
The rank function of the free elevation of $M_0$ is given by
\[
r(F)=\min\{|F\cup C_{\leq t}|-t: (C_1,\dots, C_t) \text{ is a proper $K_{d+2}$-sequence in $K_n$} \}
\]
for all $F\subseteq E(K_n)$.
\end{conjecture} 
This would imply that the free elevation of $M_0$ is the unique maximal $K_{d+2}$-matroid on $K_n$, and hence verify
Graver's conjecture that there exists a unique maximal abstract $d$-rigidity matroid.   Conjecture  \ref{con:K_d+2} holds for $d=3$ by Theorem \ref{thm:rank}. It is not difficult to see that it also holds for for $d=1,2$, see for example  \cite{JT}.

Conjecture  \ref{con:K_d+2} would imply, in particular, that
the free elevation of $M_0$ is an abstract $d$-rigidity matroid. This would follow from Theorem \ref{thm:hang} if we could show that the free elevation of $M_0$ has the 0-extension property.
%
% the problem of deciding whether a given subset of $E(K_n)$ is independent in this maximal abstract $d$-rigidity matroid is in NP$\cap$co-NP.
%
%
%Conjectures~\ref{conj:rank} and \ref{conj:max_cofactor} would imply that  the independence check problem in the maximal abstract $d$-rigidity matroid is in co-NP$\cap$NP.
% if one can show the truth of  Conjecture~\ref{conj:max_cofactor}. However there may be a much simpler trick to show this.

\item Testing generic 4-dimensional rigidity of graphs is recognized as being an even more difficult problem than generic 3-dimensional rigidity. One reason for this is Whiteley's result that ${\cal R}_{4,n}$ is not the unique maximal abstract 4-rigidity matroid since
$K_{6,6}$
%is a minimal example which
is independent in ${\cal C}_{3,n}^2$ but not in ${\cal R}_{4,n}$.
Even so, Conjecture~\ref{conj:rank} may still be robust enough to deal with such bad examples.

\begin{conjecture}\label{con:R4}
%We conjecture that 
The rank function of ${\cal R}_{4,n}$ is given by
\[
r(F)=\min\{|F\cup C_{\leq t}|-t: (C_1,\dots, C_t) \text{ is a proper $\{K_6, K_{6,6}\}$-sequence in $K_n$} \}
\]
for all $F\subseteq E(K_n)$. 
%where $(C_1,\dots, C_t)$ is said to be a $\{K_6, K_{6,6}\}$-sequence
%if each $C_i$ induces a copy of $K_6$ or $K_{6,6}$ in $K_n$.
\end{conjecture} 
This would imply that ${\cal R}_{4,n}$ is the free elevation of its rank 36 truncation, and also that ${\cal R}_{4,n}$ is  the unique maximal $\{K_6,K_{6,6}\}$-matroid on $E(K_n)$ i.e.~the unique maximal matroid in the poset of all matroids on $E(K_n)$ in which every copy of $K_6$ and $K_{6,6}$ is a  circuit.
%
%We conjecture that ${\cal R}_{4,n}$ is  the unique maximal $\{K_6,K_{6,6}\}$-matroid on $E(K_n)$ i.e.~the unique maximal matroid in the poset of all matroids on $E(K_n)$ in which every copy of $K_6$ and $K_{6,6}$ is a  circuit. Combined with Conjecture~\ref{conj:rank} (applied to the rank 36 truncation of ${\cal R}_{4,n}$), this would imply that
%the rank function 
% of ${\cal R}_{4,n}$  is given by
%\[
%r(F)=\min\{|F\cup C_{\leq t}|-t: (C_1,\dots, C_t) \text{ is a proper $\{K_6, K_{6,6}\}$-sequence in $K_n$} \}
%\]
%for all $F\subseteq E(K_n)$, where $(C_1,\dots, C_t)$ is said to be a $\{K_6, K_{6,6}\}$-sequence
%if each $C_i$ induces a copy of $K_6$ or $K_{6,6}$ in $K_n$.

\item A simpler example of a matroid in which all small circuits are copies of complete or complete bipartite graphs arises in the context of the rank 2  completion of partially filled skew-symmetric matrices, see Bernstein~\cite{B17}.
%, where 
All copies of $K_4$ and $K_{3,3}$  are circuits in this matroid and we 
%. We 
conjecture that 
it is  the unique maximal $\{K_4,K_{3,3}\}$-matroid on $E(K_n)$. Combined with Conjecture~\ref{conj:rank}, this would imply that 
its  rank function
  is given by
\[
r(F)=\min\{|F\cup C_{\leq t}|-t: (C_1,\dots, C_t) \text{ is a proper $\{K_4, K_{3,3}\}$-sequence in $K_n$} \}
\]
for all $F\subseteq E(K_n)$.
\end{enumerate}

%\newpage %% AUTHOR: please comment out this line.  It serves only
%%%   to demonstrate both types of header line in daj-template.pdf
%
%\section{Expansion estimates}
%
% More of the body of your paper goes here~\cite{bergelson-johnson-moreira}.
%
%%%% AUTHOR: optional appendix here
%\appendix %% you may comment this out if no Appendix
%\section*{Appendix}
%\section{Improving the constants}
%Material is placed here as needed.

%%% AUTHOR: optional acknowledgments here
\section*{Acknowledgments} %%  you may comment this out if no Ackno
We would like to thank Steve Noble for bringing our attention to the theory of matroid erections and for helpful conversations on this topic. We also thank the referee whose careful reading and detailed suggestions significantly improved the presentation of this paper.

This work was supported by JST CREST Grant Number JPMJCR14D2, JSPS KAKENHI Grant Number 18K11155 and EPSRC overseas travel grant EP/T030461/1.
%%% AUTHOR:
%%% Bibliography goes here. Note that the arXiv cannot process bibtex
%%% or biber bibliographies.  Example of acceptable bibliograpy format:
\bibliographystyle{amsplain}

%% AUTHOR: You can generate such a bibliography from a .bib file by 
%% running pdflatex/bibtex/pdflatex/pdflatex and then pasting the .bbl file
%% between \begin{thebibliography} and \end{bibliography}

%%% AUTHOR: Include a short description of each author following the
%%% structure below. Use the same short tags used previously.  
%%% Use \imageat{} and \imagedot{} instead of "@" and "." in
%%% email addresses-this replaces the symbols with graphics to avoid 
%%% e-mail address harvesting from the .pdf file
\begin{dajauthors}
\begin{authorinfo}[katie]
Katie Clinch\\Department of Electrical and Electronic Engineering, University of Melbourne\\ Parkville, Victoria, 3010, Australia.\\ 
%{\tt katie.clinch@unimelb.edu.au}
 katie\imagedot{}clinch\imageat{}unimelb\imagedot{}edu\imagedot{}au
\end{authorinfo}
\begin{authorinfo}[bill]
Bill Jackson\\
School
of Mathematical Sciences, Queen Mary University of London\\
Mile End Road, London E1 4NS, England.\\
b\imagedot{}jackson\imageat{}qmul\imagedot{}ac\imagedot{}uk
%{\tt b.jackson@qmul.ac.uk}  
\end{authorinfo}
\begin{authorinfo}[shin]
Shin-ichi Tanigawa\\
Department of Mathematical Informatics\\ 
Graduate School of Information Science and Technology, University of Tokyo\\ 
7-3-1 Hongo, Bunkyo-ku, 113-8656,  Tokyo, Japan. \\ 
tanigawa\imageat{}mist\imagedot{}i\imagedot{}u-tokyo\imagedot{}jp
%{\tt tanigawa@mist.i.u-tokyo.ac.jp}
  \end{authorinfo}
\end{dajauthors}

\end{document}